\documentclass[review,3p,times]{elsarticle}

\usepackage{lineno,hyperref}
\usepackage{hyperref}
\usepackage{amsfonts}

\usepackage{array,booktabs}
\usepackage{tabularx} 
\usepackage{ragged2e} 

\usepackage{amsmath,amsthm,amsfonts,amssymb,latexsym,epsfig}
\usepackage{hyperref}
\usepackage{color}
\usepackage{graphicx}
\usepackage{epstopdf}  
\usepackage{multirow}
\usepackage{float}
\usepackage{subfigure}
\usepackage{cases}

\usepackage{algorithm}
\usepackage{algorithmic}

\newtheorem{theorem}{Theorem}[section]

\newtheorem{remark}[theorem]{Remark}
\newtheorem{lemma}[theorem]{Lemma}

\newtheorem{example}[theorem]{Example}

\biboptions{sort&compress}

\numberwithin{equation}{section}

\modulolinenumbers[5]

\journal{}

\begin{document}

\begin{frontmatter}

\title{A new approach for the model updating of $(\star,\epsilon)$-palindromic systems with no spill-over\tnoteref{mytitlenote}}
\tnotetext[mytitlenote]{Supported by the Research Foundation of Education Department of Hunan Province (Grant No. 23A0266), Hunan Provincial Natural Science Foundation (Grant No. 2025JJ50034) and Natural Science Foundation of Changsha city (Grant No. kq2502074).}


\author[mymainaddress]{Kang Zhao\corref{mycorrespondingauthor}}
\cortext[mycorrespondingauthor]{Corresponding author}
\ead{zkmath@csust.edu.cn }
\address[mymainaddress]{School of Mathematics and Statistics, Changsha University of Science and Technology, Changsha, 410114, China}
\author[mymainaddress]{Xin Wang}
\author[mymainaddress]{Xiaoxiao Ma}


\begin{abstract}
In this paper, model updating problem of the  $(\star, \epsilon)$-palindromic system with no spill-over (PMUP) is considered. First, we drive the spectral decomposition of $(\star, \epsilon)$-palindromic quadratic matrix polynomial $P(\lambda)$ by a standard pair $(X,\ J)$ and a parameter matrix $\Gamma$. And the structures of $\Gamma$ are provided when $J$ is assumed to be a block diagonal matrix. Using the spectral decomposition of $P(\lambda)$,  a new sufficient solvable condition of the PMUP is provided and a set of analytical solutions is characterized.  Finally, two algorithms are developed to address the cases of both simple and multiple eigenvalues. and the performance of our proposed method is illustrated by several numerical examples.\end{abstract}



\end{frontmatter}

\section{Introduction}
In this paper, we consider the model updating problem of the following $(\star, \epsilon)$-structured palindromic quadratic system \cite{Mackey-2006b}
\begin{align}
 P(\lambda)=\lambda^2A+\lambda Q+\epsilon A^{\star}\in\mathbb{K}[\lambda],\label{gs-1}
\end{align}
where $\mathbb{K}\in\{\mathbb{R},\mathbb{C}\}$ denotes the field of real and complex numbers, $\epsilon\in\{1,\ -1\}$, $Q^{\star}=\epsilon Q$, and $\star=H$ (Hermitian) or $T$ (transpose). The palindromic quadratic eigenvalue problem (PQEP) \cite{Li-2010-NM,Chu-EKW-2010} is to find $\lambda\in\mathbb{K}$ and $x\in\mathbb{K}^n$ such that 
\begin{equation*}
P(\lambda)x=0,
\end{equation*}
where $(x,\ \lambda)$ is called an eigenpair of $P(\lambda)$. If $A$ is nonsingular, then $P(\lambda)$ has $2n$ finite nonzero eigenvalues. In this paper, we always assume that  $A$ is nonsingular.   It is well known that the eigenvalues of $P(\lambda)$ occur in pairs $(\lambda, 1/\lambda^{\star})$, or in quadruples $(\lambda, \bar{\lambda}, 1/{\lambda}, 1/{\bar{\lambda}})$ when $\mathbb{K}=\mathbb{R}$.
There are several numerical algorithms for the PQEP, such as Jacobi-type method \cite{Ipsen-2004,Mackey-2006a}, structure-preserving doubling algorithm \cite{GuoCH-siam-2010,HuangTM-siam-2009,Lu-NLAA-2015,HuangTM-2011} and palindromic QR algorithm \cite{Schroder-2008}.
The PQEP arises in a variety of applications, for example,  the vibration analysis of high-speed trains \cite{Hilliges-2004-1,Ipsen-2004}, the mathematical modeling and numerical simulation of  periodic surface acoustic wave filters \cite{Zaglmayr-2002} and the computation of the Crawford number \cite{Higham-2002}. For more examples, one can see \cite{Lin2015-book} and the references therein.   

Model updating problem (MUP), which can be seen as an inverse eigenvalue problem (IEP), is to construct a new system  using the coefficient matrices and a few ``unwanted" eigenpairs of original system, such that the prescribed or measured eigendata can be reproduced by the updated system. Generally, the unmeasured remaining eigenvalues and associated eigenvectors may changed after updating, which is called spill-over phenomenon in the literatures \cite{ChuMT-2008}. Thus, the stability of the updated system cannot be guaranteed. How to keep the remaining or unmeasured eigenapirs unchanged is of practical importance, i.e., the updated system should preserve no spill-over \cite{Chu-mc-2009,ChuMT-2008}. No spill-over MUP of the second-order systems have been intensively investigated, such as \cite{Kuo-laa-2012,Chu-mc-2009,ChuMT-2008,ChuD-siam-2009,11_2007Updating,Jia-siam-2011,Zhao-siam-2023,Saha-2025-JCAM,Mao-laa,Ganai-laa-2022}. In this paper, we consider the no spill-over MUP of $(\star, \epsilon)$-structured palindromic quadratic system (PMUP), which can be formulated as follows. 

{\bf Problem (PMUP)}
Given an $(\star, \epsilon)$-structured palindromic system $(A, Q)$ as in (\ref{gs-1}) and its $p$ eigenpairs $\{(x_j, \lambda_j)\}_{j=1}^p$. Update $(A, Q)$ to a new system $(\tilde{A}, \tilde{Q})$ with $\tilde{Q}^{\star}=\epsilon\tilde{Q}$, such that $\{(x_j, \lambda_j)\}_{j=1}^p$ are replaced by $p$ new prescribed eigenpairs $\{(\tilde{x}_j, \tilde{\lambda}_j)\}_{j=1}^p$ and the remaining $2n-p$ eigenpairs $\{(x_j, \lambda_j)\}_{j=p+1}^{2n}$ of original system $(A, Q)$ are kept unchanged. 

As is known, the MUP may be solved by the spectral decomposition of matrix polynomial.
Recently, spectral decompositions of quadratic matrix polynomials have been intensively investigated, for example, the damped symmetric system  \cite{Chu-mc-2009}, the undamped gyroscopic system \cite{Jia-siam-2011} and the vibroacoustic system \cite{Zhao-siam-2023,Qian-siam-2017}, in which the no spill-over MUP of these systems are considered. However, they did not consider the updating problem for the multiple eigenvalues. For the damped symmetric system with multiple eigenvalues, Kuo \cite{Kuo-laa-2012} provided some necessary and sufficient conditions for solvability of no spill-over MUP. However, their method can not be used to solve the PMUP. Spectral decomposition of palindromic quadratic matrix polynomial has been considered by \cite{Cai-2016}, which is used to solve the IEP of palindromic quadratic system. However, they only gave the necessary condition that a matrix pair is a standard pair of $P(\lambda)$ as in (\ref{gs-1}), and then it cannot be guaranteed that the solution of IEP obtained by their method satisfies the $\star$-symmetric structures of palindromic system. Therefore, the spectral decomposition of palindromic matrix polynomial remains open. Recently, by investigating the $\star$-symmetric solution of the matrix equation $AXB+CXD=E$, \cite{Zhao-COAM-2018} provided some sufficient solvable conditions that no spill-over updating for the $(\star, 1)$-palindromic quadratic system is possible. However, their method cannot solve the PMUP when the matrix equation is inconsistent.  The main contributions of this paper are:

(1) Spectral decomposition of the $(\star,\epsilon)$-palindromic quadratic matrix polynomial $P(\lambda)$ is derived in terms of a standard pair $(X, J)$ and a parameter matrix $\Gamma$. When $J$ is a block diagonal matrix, the structures of $\Gamma$ is provided. 

(2) A set of analytical solutions of PMUP is characterized, and two algorithms for both simple and multiple eigenvalues are proposed.

Throughout this paper, the following notations will be used. Denoted by $\mathbb{C}^{m\times n}$ and $\mathbb{R}^{m\times n}$ the set of all complex and real $m\times n$ matrices, respectively. Let $I_m$ be the $m\times m$ identity matrix, $||B||_F$ be the Frobenius norm of matrix $B$, and $\sigma(B)$ be the set of all eigenvalues of $B$.

\section{The spectral decomposition}\label{sec2}

In this section, we give a sufficient and necessary condition that a matrix pair $(X,\ J)\in\mathbb{K}^{n\times 2n}\times \mathbb{K}^{2n\times 2n}$ is a standard pair \cite{Lancaster-b-1982}  of $P(\lambda)$. Let 
\begin{equation}\label{xgx-1}
X_L=\begin{bmatrix}
X\\
XJ\\
\end{bmatrix},\ Y_L=\begin{bmatrix}
X\\
XJ^{-1}\\
\end{bmatrix}.
\end{equation}
\begin{lemma}\cite{Lancaster-b-1982}\label{lem-1}
A matrix pair $(X,\ J)$ is a standard pair of the $(\star,\epsilon)$-structure matrix polynomial $P(\lambda)$ if and only if the matrix $X_L$ defined in (\ref{xgx-1}) is nonsingular and 
\begin{equation}\label{gs-3}
AXJ^2+QXJ+\epsilon A^{\star}X=0,
\end{equation}
holds. 
\end{lemma}
For any matrix $B\in\mathbb{K}^{m\times m}$, let
\begin{equation}\label{gs-gamma}
\mathcal{S}_{(B,\star,\epsilon)}=\left\{S\in\mathbb{K}^{m\times m}|S^{\star}=-\epsilon S,\ B S=SB^{-\star}\right\}.
\end{equation}
\begin{theorem}\label{main-thm}
A matrix pair $(X,\ J)\in\mathbb{K}^{n\times 2n}\times \mathbb{K}^{2n\times 2n}$ is a standard pair of the $(\star,\epsilon)$-structure matrix polynomial $P(\lambda)$ if and only if the matrix $X_L$ defined in (\ref{xgx-1}) is nonsingular, and  there is a nonsingular matrix $\Gamma\in\mathcal{S}_{(J,\star,\epsilon)}$ satisfying
\begin{equation}\label{gs-4}
X\Gamma X^{\star}=0.
\end{equation}
And then, the coefficient matrices $A, Q$ can be given by
\begin{equation}\label{gs-5}
A=(XJ\Gamma X^{\star})^{-1},\ \ Q=-AXJ^2\Gamma X^{\star}A,
\end{equation}
with $Q^{\star}=\epsilon Q$.
\end{theorem}
\begin{proof}
(Necessity) If $(X,\ J)$ is a standard pair of $P(\lambda)$, it follows from Lemma \ref{lem-1} that (\ref{gs-3}) holds and $X_L$ is nonsingular, which implies that the matrix $Y_L$ defined in (\ref{xgx-1}) is also nonsingular. Multiplying on the left side of (\ref{gs-3}) by $X^{\star}$ leads to 
\begin{equation}\label{gs-6}
X^{\star}AXJ^2+X^{\star}QXJ+\epsilon X^{\star}A^{\star}X=0.
\end{equation}
Taking the $\star$-transpose of (\ref{gs-3}) and multiplying on the right side by $X$, we can obtain that 
\begin{equation}\label{gs-6-1}
X^{\star}A^{\star}X+\epsilon J^{-\star}X^{\star}QX+\epsilon (J^{-\star})^2X^{\star}AX=0.
\end{equation}
Let \begin{equation}\label{gs-g}
\Gamma:=\left(Y_L^{\star}\begin{bmatrix}
Q & A\\
A & 0\\
\end{bmatrix}X_L\right)^{-1}.
\end{equation}
Then, 
\begin{equation}\label{gs-7}
\Gamma^{-1}=X^{\star}QX+J^{-\star}X^{\star}AX+X^{\star}AXJ.
\end{equation}
It is easy to verify from (\ref{gs-6}) that $X^{\star}QX+X^{\star}AXJ=-\epsilon X^{\star}A^{\star}XJ^{-1}.$
Then,  (\ref{gs-7}) can be rewritten as 
$\Gamma^{-1}=J^{-\star}X^{\star}AX-\epsilon X^{\star}A^{\star}XJ^{-1},$
which implies that $\Gamma^{\star}=-\epsilon \Gamma.$ By (\ref{gs-6}) and (\ref{gs-6-1}), we can obtain that 
\begin{equation*}\begin{array}{ll}
\Gamma^{-1}J & =X^{\star}QXJ+J^{-\star}X^{\star}AXJ+X^{\star}AXJ^2=-\epsilon X^{\star}A^{\star}X+J^{-\star}X^{\star}AXJ\\
& =J^{-\star}X^{\star}QX+(J^{-\star})^2X^{\star}AX+J^{-\star}X^{\star}AXJ=J^{-\star}\Gamma^{-1},\\
\end{array}
\end{equation*}
which implies that $J\Gamma=\Gamma J^{-\star}$, i.e,  the matrix $\Gamma$ defined in (\ref{gs-g}) satisfying $\Gamma\in\mathcal{S}_{(J,\star,\epsilon)}$. It follows from (\ref{gs-g}) that 
\begin{equation*}\begin{bmatrix}
Q & A\\
A & 0\\
\end{bmatrix}^{-1}=X_L\Gamma Y_L^{\star}=\begin{bmatrix}
X\Gamma X^{\star} & X\Gamma J^{-\star}X^{\star}\\
XJ\Gamma X^{\star} & XJ\Gamma J^{-\star}X^{\star}\\
\end{bmatrix},\end{equation*}
which implies that $X\Gamma X^{\star}=0$, $A=(XJ\Gamma X^{\star})^{-1}$ and $Q=AXJ^2\Gamma X^{\star}A$, i.e.,  (\ref{gs-4}) and (\ref{gs-5}) are satisfied. 

(Sufficiency) If there exists a nonsingular matrix $\Gamma\in\mathcal{S}_{(J,\star,\epsilon)}$ such that (\ref{gs-4}) holds, then we have 
\begin{equation}\label{gs-8}
X_L\Gamma Y_L^{\star}=\begin{bmatrix}
X\\
XJ\\
\end{bmatrix}\Gamma\begin{bmatrix}
X^{\star} & J^{-\star}X^{\star}\\
\end{bmatrix}=\begin{bmatrix}
0 & X\Gamma J^{-\star}X^{\star}\\
XJ\Gamma X^{\star} & XJ\Gamma J^{-\star}X^{\star}\\
\end{bmatrix}.
\end{equation}
Since $X_L$ and $Y_L$ are nonsingular, it follows from (\ref{gs-8}) that the matrix $XJ\Gamma X^{\star}$ is also nonsingular, which implies that the matrix $A$ given by (\ref{gs-5}) is well-defined. Substituting the matrices $A$ and $Q$ given by (\ref{gs-5}) into (\ref{gs-8}) leads to 
\begin{equation}\label{gs-9}
\Gamma^{-1}=Y_L^{\star}\begin{bmatrix}
Q & A\\
A & 0\\
\end{bmatrix}X_L,
\end{equation}
 which means that the matrix $\Gamma$ has the form as in (\ref{gs-g}). Recall that $\Gamma^{\star}=-\epsilon \Gamma$ and $J\Gamma=\Gamma J^{-\star}$, it follows that
$\Gamma J^{\star}=J^{-1}\Gamma.$
 Taking $\star$-transpose of the matrix $A$ given in (\ref{gs-5}), we have  
 \begin{equation}\label{gs-10}
 A^{\star}=(X\Gamma^{\star}J^{\star}X^{\star})^{-1}=-\epsilon(XJ^{-1}\Gamma X^{\star})^{-1}.
 \end{equation}
Note that 
\begin{equation}\label{gs-12}
X_LJ^{-1}\Gamma Y_L^{\star}=\begin{bmatrix}
X\\
XJ\\
\end{bmatrix}J^{-1}\Gamma\begin{bmatrix}
X^{\star} & J^{-\star}X^{\star}\\
\end{bmatrix}=
\begin{bmatrix}
XJ^{-1}\Gamma X^{\star} & 0\\
0 & X\Gamma J^{-\star}X^{\star}\\
\end{bmatrix}.
\end{equation}
Substituting the matrices $A$ and $A^{\star}$ given by (\ref{gs-5}) and (\ref{gs-10}) into (\ref{gs-12}), we can obtain that
\begin{equation}\label{gs-13}
\Gamma^{-1}J=Y_L^{\star}\begin{bmatrix}
-\epsilon A^{\star} & 0\\
0 & A\\
\end{bmatrix}X_L.
\end{equation}
 Combining (\ref{gs-9}) and (\ref{gs-13}), we have 
\begin{equation*}
AXJ^2+QXJ+\epsilon A^{\star}X=0,\ \ \
AXJ^2+\epsilon Q^{\star}XJ+\epsilon A^{\star}X=0,
\end{equation*}
which implies that (\ref{gs-3}) holds and $\epsilon Q^{\star}XJ=QXJ$. It follows that $Q^{\star}=\epsilon Q$, since $J$ is nonsingular and the matrix $X\in\mathbb{K}^{n\times 2n}$ is of full row rank, i.e., the matrix $Q$ given in (\ref{gs-5}) is well-defined. 
\end{proof}

\section{Structures of $\mathcal{S}_{(J,\star,\epsilon)}$}\label{sec3}

In this section, we characterize a standard pair $(X,\ J)$ of $P(\lambda)$ and give the structure of the matrix $\Gamma\in\mathcal{S}_{(J,\star,\epsilon)}$. First, we give some spectral structures of the palindromic matrix polynomials. 

\begin{lemma}\label{lem-6}
Let $\mathbb{K}=\mathbb{R}$ and $\lambda_1,\lambda_2,\ldots,\lambda_{2n}$ be the eigenvalues of $(T,\epsilon)$-palindromic $P(\lambda)$ as in (\ref{gs-1}). Then

(1) $P(\lambda)$ is $T$-anti-palindromic with odd degree:  both of $1$ and $-1$ are eigenvalues of $P(\lambda)$, which have odd algebraic multiplicities. Moreover, $\Pi_{k=1}^{2n}\lambda_k=-1.$

(2) $P(\lambda)$ is $T$-anti-palindromic with even degree ($T$-palindromic):  if $1$ (or $-1$) is eigenvalue of $P(\lambda)$, then its algebraic multiplicity must be even. And all the eigenvalues of $P(\lambda)$ satisfy  
$\Pi_{k=1}^{2n}\lambda_k=1.$
\end{lemma}

\begin{proof}
Since $A$ is nonsingular, we can see from (\ref{gs-1}) that 
\begin{equation}\label{gsg-2}
\det(P(\lambda))=(-1)^n\lambda^{2n}\det\left(P(1/{\lambda})\right):=f(\lambda).
\end{equation}

(1) For $T$-anti-palindromic, it is easy to see that both of $A+Q-A^T$ and $A-Q-A^T$ are skew-symmetric matrices. Since $n$ is odd, it follows that $f(1)=f(-1)=0$, i.e., both of $1$ and $-1$ are eigenvalues of $P(\lambda)$. Suppose that the algebraic multiplicity of $1$ is $m$, then there exists a polynomial $h(\lambda)\in\mathbb{R}[\lambda]$ with $h(1)\neq 0$ such that 
\begin{equation}\label{gsg-3}
f(\lambda)=(\lambda-1)^mh(\lambda).
\end{equation}
Note that,
\begin{equation*}f(1/\lambda)=\left(\frac{1}{\lambda}-1\right)^mh(1/\lambda)=(-1)^m(\lambda-1)^m\lambda^{-m}h(1/\lambda).
\end{equation*}
Since $n$ is odd, it follows from (\ref{gsg-2}) and (\ref{gsg-3}) that 
$(\lambda-1)^mh(\lambda)=(-1)^m(\lambda-1)^{m+1}\lambda^{2n-m}h(1/\lambda),$
which implies that 
\begin{equation}\label{gsg-4}
h(\lambda)=(-1)^{m+1}\lambda^{2n-m}h(1/\lambda).
\end{equation}
Setting $\lambda=1$ in (\ref{gsg-4}) leads to $(-1)^{m+1}=1$, i.e., $m$ is odd. Similarly, the algebraic multiplicity of $-1$ is odd. Since all eigenvalues occur in quadruples, it follows that $\Pi_{k=1}^{2n}\lambda_k=-1.$ The case of (2) can be proved similarly.
\end{proof}

 Let $N_j$ be an $n_j\times n_j$ nilpotent matrix with ones or zeros along its superdiagonal.  Next, we provide the Jordan canonical form $J_j$ according to the distinct eigenvalue $\lambda_j$.

Case 1. $(T, \epsilon)$-palindromic with $\mathbb{K}=\mathbb{R}$
\begin{itemize}
  \item $J_j=\mbox{diag}\left(\lambda_jI_{n_j}+N_{n_j}, \bar{\lambda}_jI_{n_j}+N_{n_j}, \frac{1}{\lambda_j}I_{n_j}+N_{n_j}, \frac{1}{\bar{\lambda}_j}I_{n_j}+N_{n_j}\right)$, $\lambda_j\in\mathbb{C}/\mathbb{R}$, $|\lambda_j|\neq 1,$ $j=1,\ldots,r_1$,
  \item $J_j=\mbox{diag}\left(\lambda_jI_{n_j}+N_{n_j},\ \bar{\lambda}_jI_{n_j}+N_{n_j}\right)$, $\lambda_j\in\mathbb{C}/\mathbb{R}$,\ $|\lambda_j|=1,$\ $j=r_1+1,\ldots, r_2$,
  \item $J_j=\mbox{diag}\left(\lambda_jI_{n_j}+N_{n_j}, \frac{1}{\lambda_j}I_{n_j}+N_{n_j}\right)$, $\lambda_j\in\mathbb{R}$, $\lambda_j^2\neq 1$, $j=r_2+1,\ldots,r-2$,
  \item $J_{r-1}=I_{n_{r-1}}+N_{n_{r-1}}$,
  \item $J_r=-I_{n_r}+N_{n_r},$
\end{itemize}
where $4(n_1+\cdots+n_{r_1})+2(n_{r_1+1}+\cdots+n_{r-2})+n_{r-1}+n_r=2n$.

Case 2. $(H, \epsilon)$-palindromic with $\mathbb{K}=\mathbb{C}$
\begin{itemize}
\item $J_j=\mbox{diag}\left(\lambda_jI_{n_j}+N_{n_j}, \frac{1}{\bar{\lambda}_j}I_{n_j}+N_{n_j}\right)$, $\lambda_j\in\mathbb{C}$, $|\lambda_j|\neq 1, j=1,\ldots,l$;
\item $J_j=\lambda_jI_{n_j}+N_{n_j},\ \lambda_j\in\mathbb{C}, |\lambda_j|= 1,\  j=l+1,\ldots,r$,
\end{itemize}
where $2(n_1+\cdots+n_{l}+n_{l+1}+\cdots+n_r)=2n$.

 Let $(X,\ J)\in\mathbb{C}^{n\times 2n}\times \mathbb{C}^{2n\times 2n}$ be of forms 
\begin{align}
& X:=[X_1,X_2,\ldots,X_r],\ \ \ J:=\mbox{diag}(J_1,J_2,\ldots,J_r),\label{gs-14-1}
\end{align}
where $X_j$ be the $n\times n_j$ matrix whose columns form the corresponding generalized eigenspace of $J_j$. Clearly, $J$ is complex Jordan canonical form of the $(\star,\epsilon)$-palindromic polynomial $P(\lambda)$.
As is known, the matrix 
\begin{equation}\label{xgx-12}
\begin{bmatrix}
\lambda I_k+N_k & 0\\
0 & \bar{\lambda}I_k+N_k\\
\end{bmatrix}\in\mathbb{C}^{2k\times 2k}
\end{equation} is similar to 
\begin{equation*}\begin{bmatrix}
D(\lambda) & I_2 & 0 &  & 0\\
           & D(\lambda) & I_2 & & \\
           & & & \ddots& \\
           & 0 & \ddots & & I_2\\
           & & & & D(\lambda)\\
           \end{bmatrix},
           \end{equation*}
via a permutation matrix, which has $k$ diagonal blocks $D(\lambda)=\mbox{diag}(\lambda,\bar{\lambda})$ on its main diagonal and $k-1$ identity matrices $I_2$ on its superdiagonal. Clearly, 
$\mathcal{P}D\mathcal{P}^H=\bigr[\begin{smallmatrix}
\alpha & \beta\\
-\beta & \alpha\\
\end{smallmatrix}\bigr],$
where $\lambda=\alpha+\beta{\rm i}$ and $\mathcal{P}=\frac{1}{\sqrt{2}}\bigl[\begin{smallmatrix}
1 & 1\\
{\rm i} & -{\rm i}\\
\end{smallmatrix}\bigr]$ is an unitary matrix. Therefore,  It is easy to verify that the complex Jordan canonical form (\ref{xgx-12}) is unitary similar to 
\begin{equation}\label{xgx-13}
C_k(\alpha,\beta)=\begin{bmatrix}
C(\alpha,\beta) & I_2 & 0\\
 & C(\alpha,\beta) & \ddots & \\
 & & \ddots &I_2\\
 0 & & & C(\alpha,\beta)\\
 \end{bmatrix}.
\end{equation}
Moreover, if $x:=x_R+{\rm i}x_I\in\mathbb{C}^n$ is an eigenvector of $\lambda$, then $[x, \bar{x}]\mathcal{P}^H=[\sqrt{2}x_r,\ \sqrt{2}x_I]\in\mathbb{R}^n$. Note that, all eigenpairs of $(T,\epsilon)$-palindromic $P(\lambda)$ are closed under complex conjugation when $\mathbb{K}=\mathbb{R}$. Then, the real-valued Jordan canonical form of  $J$ can be given by the following lemma, which can also be found in \cite{Roger-book-2012}.

\begin{lemma}\label{remark-1}
Let $\mathbb{K}=\mathbb{R}$  and $\star=T$. Suppose that $(X,\ J)\in\mathbb{C}^{n\times 2n}\times\mathbb{C}^{2n\times 2n}$ given by (\ref{gs-14-1}) is a standard pair of $P(\lambda)$ as in (\ref{gs-1}), where $J$ is a complex Jordan canonical form of $P(\lambda)$. Then, there exists an $2n\times 2n$ unitary matrix $M$ of form 
\begin{equation}\label{xg-m}
M=\mbox{diag}\left(\frac{1}{2}\begin{bmatrix}
1 & 1\\
{\rm i} & -{\rm i}\\
\end{bmatrix},\ldots,\frac{1}{2}\begin{bmatrix}
1 & 1\\
{\rm i} & -{\rm i}\\
\end{bmatrix},I_{2n-m}\right)Q,
\end{equation}
where $Q\in\mathbb{R}^{2n\times 2n}$ is a permutation matrix, $m=4(n_1+\cdots+n_{r_1})+2(n_{r_1+1}+\ldots,n_{r_2})$, such that 
$(XM^H,\ J_R)\in\mathbb{R}^{n\times 2n}\times \mathbb{R}^{2n\times 2n}$ is a real-valued standard pair of $P(\lambda)$, where $J_R:=MJM^H$ is the real-valued Jordan canonical form of $P(\lambda)$, which is of the following form
\begin{equation*}
J_R=\mbox{diag}\left(C_{n_1}(\alpha_1,\beta_1),\ldots,C_{n_t}(\alpha_{r_2},\beta_{r_2}), J_{n_{r_2+1}}(\lambda_r),\ldots,J_{n_r}(\lambda_r)\right)\in\mathbb{R}^{2n\times 2n},
\end{equation*}
in which $\lambda_k=\alpha_k+\beta{\rm i}\in\mathbb{C}/\mathbb{R}$, $k=1,\ldots,r_2$, and $J_{n_k}(\lambda_k)$ is the Jordan canonical form of the real eigenvalue $\lambda_k$, $k=r_2+1,\ldots,r.$ 
\end{lemma}

\begin{lemma}\label{lem-2}
Suppose that $(X,\ J)$ defined by (\ref{gs-14-1}) is a standard pair of $(\star,\epsilon)$-palindromic $P(\lambda)$, then there exists a matrix $\Gamma\in\mathcal{S}_{(J,\star,\epsilon)}$ such that (\ref{gs-5}) holds and 
\begin{equation}\label{gs-15}
\Gamma=\mbox{diag}(\Gamma_{11},\Gamma_{22},\ldots,\Gamma_{rr}),
\end{equation}
where $\Gamma_{jj}\in\mathcal{S}_{(J_j,\star,\epsilon)}$, $j=1,\ldots,r.$
\end{lemma}
\begin{proof}
Since $(X,\ J)$ is a standard pair of $P(\lambda)$, we can see from Theorem \ref{main-thm} that there exists a nonsingular matrix $\Gamma\in\mathcal{S}_{(J,\star,\epsilon)}$ such that (\ref{gs-4}) and (\ref{gs-5}) are satisfied. It follows from (\ref{gs-gamma}) that 
\begin{equation}\label{gs-16-1}
\Gamma^{\star}=-\epsilon\Gamma,\ \  J\Gamma=\Gamma J^{-\star}.
\end{equation}
Partition the matrix $\Gamma$ as $\Gamma:=[\Gamma_{ij}]_{r\times r}$ according to the matrix $J$ given in (\ref{gs-14-1}), and we can obtain from (\ref{gs-16-1}) that 
\begin{equation}\label{gs-17}
J_l\Gamma_{lj}-\Gamma_{lj}J_j^{-\star}=0,\ j, l=1,2,\ldots, r,
\end{equation}
which implies that
\begin{equation}\label{gs-18}
\left[I\otimes J_l-(J_j^{-\star})^T\otimes I\right]\mbox{\bf vec}(\Gamma_{lj})=0,
\end{equation}
where $\otimes$ is the Kronecker product and $\mbox{\bf vec}$ is the column vectorization of a matrix. According to the structures of the Jordan canonical forms provided in Case 1 and Case 2, it is easy to verify that the matrix $I\otimes J_l-(J_j^{-\star})^T\otimes I$ is nonsingular when $j\neq l$, which implies that $\mbox{\bf vec}(\Gamma_{lj})=0$, i.e., $\Gamma_{lj}=0$ $(j\neq l)$. By (\ref{gs-16-1}), it follows from the definition (\ref{gs-gamma}) that $\Gamma_{jj}\in\mathcal{S}_{(J_j,\star,\epsilon)}$, $j=1,\ldots, r.$
\end{proof}

By Lemma \ref{lem-2}, we can see that the structure of $\mathcal{S}_{(J,\star,\epsilon)}$ need to be considered according to the Jordan canonical forms $J_j$ associated with distinct eigenvalue $\lambda_j$.  
Without loss of generality, we can assume that there are $m_j$ Jordan blocks in the Jordan canonical form $J_j$. And then, the nilpotent matrix $N_j$ can be expressed as 
\begin{equation}\label{gs-19-1}
N_j=\mbox{diag}\left(N_1^{(j)}, N_2^{(j)}, \ldots, N_{m_j}^{(j)}\right),
\end{equation}
where $N_t^{(j)}$ is the nilpotent block of order $n_t^{(j)}\times n_t^{(j)}$ for $t=1,\ldots, m_j$. Let
\begin{equation}\label{gs-19}
\Theta_{(a,b,j)}=\left\{Z\in\mathbb{C}^{n_j\times n_j}| aZN_j^T+bN_jZ+N_{j}ZN_{j}^T=0\right\},
\end{equation}
where $a, b\in\mathbb{C}$ are nonzero. A straightforward calculation shows that any matrix $Z\in\Theta_{(a,b,j)}$ is necessarily of the form
\begin{equation}\label{gs-28}
Z=\begin{bmatrix}
Z_{11} & Z_{12} & \cdots & Z_{1m_j}\\
\vdots & \vdots & & \vdots \\
Z_{m_j1} & Z_{m_j2} & \cdots & Z_{m_jm_j}\\
\end{bmatrix},\ \ Z_{ik}\in\mathbb{C}^{n_i^{(j)}\times n_k^{(j)}},
\end{equation} 
where $Z_{ii}\in\mathbb{C}^{n_i^{(j)}\times n_i^{(j)}}$ is an anti-lower triangular matrix (whose (s,t)-element is zero if $s+t\geq n_i^{(j)}+2$) and $Z_{ik}\in\mathbb{C}^{n_i^{(j)}\times n_k^{(j)}}$ ($i\neq k$) is a matrix that the $(s,t)$-element of $Z_{ik}$ is zero if $s+t>\max\{n_i^{(j)}, n_k^{(j)}\}$.

\subsection{The general case}
\begin{theorem}\label{thm-gamma}
Suppose that $(X,\ J)$ defined by (\ref{gs-14-1}) is a standard pair of $(T, \epsilon)$-palindromic $P(\lambda)$ with $\mathbb{K}=\mathbb{R}$, then there exists an unitary matrix $M\in\mathbb{C}^{2n\times 2n}$ of form (\ref{xg-m}) and  $\Gamma\in\mathcal{S}_{(J,H,\epsilon)}\in\mathbb{C}^{2n\times 2n}$ such that the coefficient matrices $A, Q\in\mathbb{R}^{n\times n}$ of $P(\lambda)$ can be expressed as
\begin{equation}\label{gs-27}
A=(X_RJ_R\Gamma_R X_R^{T})^{-1},\ \ Q=-AX_RJ_R^2\Gamma_R X_R^{T}A,
\end{equation}
where $Q^T=\epsilon Q$, $X_R=XM^{H}\in\mathbb{R}^{n\times 2n}$ and $J_R=MJM^{H}\in\mathbb{R}^{2n\times 2n}$ is real-valued Jordan canonical form of $P(\lambda)$. Moreover, the matrix $\Gamma_R$ has the following form 
\begin{equation}\label{gs-26}
\Gamma_R=M\mbox{diag}\left(\Phi_{1},\Phi_2,\ldots, \Phi_{r}\right)M^{T},
\end{equation}
where
\begin{align}
&\Phi_{j}=\begin{bmatrix}
0 & V_j\\
-\epsilon V_j^T & 0\\
\end{bmatrix},\ \ V_j=\begin{bmatrix}
V_{j1} & 0\\
0 & V_{j4}\\
\end{bmatrix}, \ j=1,\ldots, r_1,\label{gs-g-1}\\
&\Phi_j=\begin{bmatrix}
0 & F_j\\
-\epsilon F_j^T & 0\\
\end{bmatrix},\ \ j=r_1+1,\ldots, r_2,\label{gs-g-2}\\
&\Phi_j=\begin{bmatrix}
0 & E_j\\
-\epsilon E_j^T & 0\\
\end{bmatrix},\ \ j=r_2+1,\ldots, r-2,\label{gs-g-3}
\end{align}
with $V_{j1}\in\Theta_{(\lambda_j, 1/\lambda_j, j)},$ $V_{j4}\in\Theta_{(\bar{\lambda}_j, 1/\bar{\lambda}_j, j)}$, $F_j\in\Theta_{(\lambda_j, \bar{\lambda}_j, j)}$, $E_j\in\Theta_{(\lambda_j, 1/\lambda_j, j)}$, and $\Phi_{r-1}\in\Theta_{(1,1, r-1)}$, $\Phi_{r}\in\Theta_{(-1,-1, r)}$ with $\Phi_{r-1}^T=-\epsilon \Phi_{r-1}$, $\Phi_{r}^T=-\epsilon \Phi_{r}$.
\end{theorem}
\begin{proof}
We only given the proof for the case $\epsilon=1$, and the case of $\epsilon=-1$ can be proved similarly. We can see from Lemma \ref{remark-1} that there is an unitary matrix $M\in\mathbb{C}^{2n\times 2n}$ of form (\ref{xg-m}) such that $(X_R,\ J_R)$ is a real standard pair of $P(\lambda)$, where $X_R=XM^{H}$ and $J_R=MJM^{H}$ is the real Jordan canonical form. It follows from Theorem \ref{main-thm} that there is a nonsingular $\Gamma_R\in\mathcal{S}_{(J_R, T, 1)}$ such that (\ref{gs-27}) holds. Then, we can see from (\ref{gs-gamma}) that $\Gamma_R^T=-\Gamma_R$ and 
\begin{equation}\label{gs-20}
J_R\Gamma_R J_R^T-\Gamma_R=0.
\end{equation}
Substituting $J_R=MJM^{H}$ into (\ref{gs-20}) leads to 
\begin{equation}\label{gs-21}
J\Phi J^T-\Phi=0,
\end{equation}
where $\Phi=M^{H}\Gamma_R (M^H)^T\in\mathbb{C}^{2n\times 2n}$ and $\Phi^T=-\Phi$. By Lemma \ref{lem-2}, we have 
$\Phi=\mbox{diag}(\Phi_1,\ldots, \Phi_r),$
where $\Phi_j$ satisfies 
\begin{equation}\label{gs-23}
J_j\Phi_j J_j^T-\Phi_j=0,\ j=1,\ldots, r.
\end{equation}

Next, we give the proof of the case that $\lambda_j\in\mathbb{C}/\mathbb{R}$ with $|\lambda_j|\neq 1$, $j=1,\ldots, r_1$, and the other cases can be proved similarly. 
Partition $\Phi_j$ according to the block structure of $J_j$ given by Case 1 as 
\begin{equation*}\Phi_j=\begin{bmatrix}
U_{j1} & U_{j2} & V_{j1} & V_{j2}\\
- U_{j2}^T & U_{j3} & V_{j3} & V_{j4}\\
- V_{j1}^T & - V_{j3}^T & W_{j1} & W_{j2}\\
- V_{j2}^T & - V_{j4}^T & - W_{j2}^T & W_{j3}\\
\end{bmatrix}.\end{equation*}
Then, we can obtain from (\ref{gs-23}) that 
\begin{align}
&(\lambda_j^2-1)U_{j1}+\lambda_jU_{j1}N_j^T+\lambda_j N_jU_{j1}+N_jU_{j1}N_j^T=0,\label{gs-24-1}\\
&(\lambda_j\bar{\lambda}_j-1)U_{j2}+{\lambda}_jU_{j2}N_j^T+\bar{\lambda}_j N_jU_{j2}+N_jU_{j2}N^T=0,\label{gs-24-2}\\
&(\bar{\lambda}_j^2-1)U_{j3}+\bar{\lambda}_jU_{j3}N_j^T+\bar{\lambda}_j N_jU_{j3}+N_jU_{j3}N_j^T=0,\label{gs-24-3}\\
&\lambda_jV_{j1}N_j^T+\frac{1}{\lambda_j}N_jV_{j1}+N_jV_{j1}N_j^T=0,\label{gs-24-4}\\
&\left(\frac{\lambda_j}{\bar{\lambda}_j}-1\right)V_{j2}+\lambda_jV_{j2}N_j^T+\frac{1}{\bar{\lambda}_j}V_{j2}+N_jV_{j2}N_j^T=0,\label{gs-24-5}\\
&\left(\frac{\bar{\lambda}_j}{{\lambda}_j}-1\right)V_{j3}+\bar{\lambda}_jV_{j3}N_j^T+\frac{1}{\lambda_j}N_jV_{j3}+N_jV_{j3}N_j^T=0,\label{gs-24-6}\\
&\bar{\lambda}_jV_{j4}N_j^T+\frac{1}{\bar{\lambda}_j}N_jV_{j4}+N_jV_{j4}N_j^T=0.\label{gs-24-7}
\end{align}
It is easy to verify from (\ref{gs-24-1}) that 
\begin{equation}\label{gs-25}
\left((\lambda_j^2-1)I_{n_j}\otimes I_{n_j}+N_j\otimes \lambda_jI_{n_j}+\lambda_jI_{n_j}\otimes N_j+N_j\otimes N_j\right){\mbox{\bf vec}}(U_{j1})=0.
\end{equation}
Since $\lambda_j\in\mathbb{C}/\mathbb{R}$ and $|\lambda_j|\neq 1$, the matrix 
\begin{equation*}(\lambda_j^2-1)I_{n_j}\otimes I_{n_j}+N_j\otimes \lambda_jI_{n_j}+\lambda_jI_{n_j}\otimes N_j+N_j\otimes N_j
\end{equation*} 
is nonsingular. It follows from (\ref{gs-25}) that $U_{j1}=0$. Similarly, we can obtain from (\ref{gs-24-2}), (\ref{gs-24-3}), (\ref{gs-24-5}) and (\ref{gs-24-6}) that $U_{j2}=0, U_{j3}=0, V_{j2}=0$ and $V_{j3}=0$, respectively. From (\ref{gs-24-4}) and (\ref{gs-24-7}), it is easy to verify that $V_{j1}\in\Theta_{(\lambda_j, 1/\lambda_j, j)}$ and $V_{j4}\in\Theta_{(\bar{\lambda}_j, 1/\bar{\lambda}_j, j)}.$ 
In a similar way, we can prove that $W_{j1}=0, W_{j2}=0$ and $W_{j3}=0$.
\end{proof}

\begin{theorem}\label{thm-gamma-2}
Suppose that $(X,\ J)$ given by (\ref{gs-14-1})  is a standard pair of $(H, \epsilon)$-palindromic $P(\lambda)$ with $\mathbb{K}=\mathbb{C}$, then there is a matrix $\Gamma\in\mathcal{S}_{(J,H,\epsilon)}\in\mathbb{C}^{2n\times 2n}$ such that the coefficient matrices $A, Q\in\mathbb{C}^{n\times n}$ of $P(\lambda)$ can be expressed as
\begin{equation}\label{gs-36}
A=(XJ\Gamma X^{H})^{-1},\ \ Q=-AXJ^2\Gamma X^{H}A,
\end{equation}
with $Q^H=\epsilon Q$,  and the matrix $\Gamma$ has the following form 
\begin{equation}\label{gs-38}
\begin{array}{ll}
\Gamma=&\mbox{diag}\left(\begin{bmatrix}
0 & S_1\\
-\epsilon S_1^H & 0\\
\end{bmatrix},\cdots,\begin{bmatrix}
0 & S_{l}\\
-\epsilon S_{l}^H & 0\\
\end{bmatrix},S_{l+1},\ldots,S_r\right),\\
\end{array}
\end{equation}
where $S_j\in\Theta_{(\lambda_j, 1/\lambda_j, j)}$ for $j=1,\ldots, l$ and $S_{j}\in\Theta_{(\lambda_j, \bar{\lambda}_j, j)}$ with $S_{j}^H=-\epsilon S_{j}$ for $j=l+1,\ldots, r$.  
\end{theorem}
\begin{proof}
The proof of this theorem is similar to the Theorem \ref{thm-gamma}, so we omit the details here.
\end{proof}

\subsection{The case of semi-simple}

 As is known, an eigenvalue  is called semi-simple if its algebraic multiplicity is equal to geometric multiplicity. When all eigenvalues of $P(\lambda)$ are semi-simple, the matrices $\Gamma_R$ for $(T, \epsilon)$-palindromic and $\Gamma$ for $(H, \epsilon)$-palindromic still have the forms as in (\ref{gs-26}) and (\ref{gs-38}), respectively, but the special structures as in $\Theta_{(a, b, j)}$ will no longer show up. In fact, the structures of $\Gamma_R$ in (\ref{gs-26}) and $\Gamma $ in (\ref{gs-38}) can be further simplified if all the eigenvalues of $P(\lambda)$ are semi-simple. The following lemma can be obtained directly from the SVD \cite{Roger-book-2012} of a matrix.

 \begin{lemma}\label{lem-3} 
 Let $W=\begin{bmatrix}
 0 & W_1\\
 -\epsilon W_1^T & 0 \\
 \end{bmatrix}$ be a nonsingular matrix, where $W_1=\mbox{diag}(W_{11},W_{22})$, $W_{11}, W_{22}\in\mathbb{C}^{n\times n}$. There exist unitary matrices $U_1, U_2, V_1, V_2\in\mathbb{C}^{n\times n}$ such 
 \begin{equation}\label{gs-30}
 PWP^T=\begin{bmatrix}
 0 & D\\
 -\epsilon D & 0\\
 \end{bmatrix},\ \ \ D=\mbox{diag}(\sigma_1,\ldots, \sigma_n),\ \ \sigma_j>0,
 \end{equation}
 where $P=\mbox{diag}(U_{1}^H,\ U_{2}^H, V_{1}^T, V_2^T)$.
 \end{lemma}

\begin{theorem}\label{thm-s-1}
Let $(X,\ J)$ defined by (\ref{gs-14-1}) be a standard pair of $(T, \epsilon)$-palindromic $P(\lambda)$ with $\mathbb{K}=\mathbb{R}$. Suppose that all eigenvalues of $P(\lambda)$ are semi-simple. Then there exists an unitary matrix $M\in\mathbb{C}^{2n\times 2n}$  of form (\ref{xg-m}) and a nonsingular matrix $T\in\mathbb{C}^{2n\times 2n}$ such that $(Y_R,\ J_R)$ is a real-valued standard pair of $P(\lambda)$, where $Y_R=XTM^H$, $J_R=MJM^{H}$ and the corresponding matrix $\tilde{\Gamma}_R\in\mathcal{S}_{(J_R,T,\epsilon)}$ is of form
\begin{equation}\label{gs-31}
\tilde{\Gamma}_R=M\mbox{diag}\left(\begin{bmatrix}
0 & I_{2n_1}\\
-\epsilon I_{2n_1} & 0\\
\end{bmatrix},\cdots,\begin{bmatrix}
0 & I_{2n_{r_1}}\\
-\epsilon I_{2n_{r_1}} & 0\\
\end{bmatrix}, \begin{bmatrix}
0 & I_{n_{r_1+1}}\\
-\epsilon I_{n_{r_1+1}} & 0\\
\end{bmatrix},\cdots,\begin{bmatrix}
0 & I_{n_{r-2}}\\
-\epsilon I_{n_{r-2}} & 0\\
\end{bmatrix}, \mathfrak{I}_{r-1}, \mathfrak{I}_r\right)M^T,\\
\end{equation}
where $4(n_1+\cdots+n_{r_1})+2(n_{r_1+1}+\cdots+n_{r-2})+n_{r-1}+n_r=2n$ and for $k=n_{r-1}, n_r$,
\begin{align}
&\mathfrak{I}_k=\begin{bmatrix}
0 & I_{k/2}\\
-\epsilon I_{k/2} &0\\
\end{bmatrix},\  \mbox{if}\ \epsilon=1\ \mbox{or}\ \epsilon=-1\ \mbox{and}\ n \ \mbox{is\ even},\label{xg-6-1}\\
&\mathfrak{I}_k=I_{k},\ \mbox{if}\ \epsilon=-1\ \mbox{and}\ n\ \mbox{is\ odd}.\label{xg-6-2}
\end{align}

\end{theorem} 
\begin{proof}
We only give the proof of the case $\epsilon=1$, and the case of $\epsilon=-1$ can be proved similarly. Similar to the proof of Theorem \ref{thm-gamma}, we can obtain that  $(X_R,\ J_R)$ is a real-valued standard pair of $P(\lambda)$, where  $X_R=XM^H$ and $J_R=MJM^{H}$. And then $P(\lambda)$ has a real-valued  spectral decomposition given by (\ref{gs-27}) in which $\Gamma_R\in\mathcal{S}_{(J_R,T,\epsilon)}$ has form as given by (\ref{gs-26}). Now, we discuss the structure of $\Phi_j$ according to the Jordan canonical form $J_j$ for distinct eigenvalue $\lambda_j$.

(1). $\lambda_j\in\mathbb{C}/\mathbb{R}$ with $|\lambda_j|\neq 1$, and $\Phi_j$ is given by (\ref{gs-g-1}), for $j=1,\ldots,r_1$. By Lemma \ref{lem-3}, there exist unitary matrices $Q_{j1},Q_{j2},Q_{j3},Q_{j4}\in\mathbb{C}^{n_j\times n_j}$  such that 
\begin{equation*}Q_j\begin{bmatrix}
0 & V_j\\
-\epsilon V_j^T & 0\\
\end{bmatrix}Q_j^T=\begin{bmatrix}
0 & D_j\\
-\epsilon D_j & 0\\
\end{bmatrix},\end{equation*}
where $Q_j=\mbox{diag}(Q_{j1}^H,Q_{j2}^H,Q_{j3}^T,Q_{j4}^T)$, $D_j=\mbox{diag}(\sigma_1^{(j)},\ldots,\sigma_{2n_j}^{(j)}),$ $\sigma_k^{(j)}>0$, $k=1,\ldots, 2n_j$. Let $D_{jj}=\mbox{diag}(D_j, D_j)^{1/2}$. Then
\begin{equation}\label{xg-3}
D_{jj}^{-1}Q_j\Phi_jQ_j^TD_{jj}^{-1}=\begin{bmatrix}
0 & I_{2n_j}\\
-\epsilon I_{2n_j} & 0\\
\end{bmatrix}, \ j=1,\ldots, r_1.
\end{equation}
Let $\hat{X}_j:=X_jQ^{-1}D_{jj}$. Since $N_j=0$, it follows that \begin{equation}\label{xg-1}
\hat{X}_jJ_j=X_jQ^{-1}D_{jj}J_j=X_j Q^{-1}J_jD_{jj}=X_jJ_jQ^{-1}D_{jj},\ j=1,\ldots,r_1.
\end{equation}

(2). Similarly, for Cases of $\lambda_j\in\mathbb{C}/\mathbb{R}$ with $|\lambda_j|=1$, $j=r_1+1,\ldots,r_2$ and $\lambda_j\in\mathbb{R}$ with $\lambda_j^2\neq 1$, $j=r_2+1,\ldots,r-2$, we can prove that there exists an unitary matrix $Q_j$ and a nonsingular diagonal matrix $D_{jj}$ such that 
\begin{equation}\label{xg-4}
D_{jj}^{-1}Q_j\Phi_jQ_j^TD_{jj}^{-1}=\begin{bmatrix}
0 & I_{n_j}\\
-\epsilon I_{n_j} & 0\\
\end{bmatrix}, \ j=r_1+1,\ldots,r-2.
\end{equation}
Let $\hat{X}_j:=X_jQ^{-1}D_{jj}$. It is easy to see that
 \begin{equation}\label{xg-2}
\hat{X}_jJ_j=X_jQ^{-1}D_{jj}J_j=X_j Q^{-1}J_jD_{jj}=X_jJ_jQ^{-1}D_{jj}, \ j=r_1+1,\ldots,r-2.
\end{equation}

(3) $\lambda_{{r-1}}=1$, $\lambda_{r}=-1$ and $\Phi_{r-1}^T=-\Phi_{r-1}$, $\Phi_r^T=-\Phi_r$.  By Lemma \ref{lem-6}, we can obtain that the numbers $n_{r-1}$ and $n_r$ must be even. There exists an unitary matrix $Q_{j}\in\mathbb{R}^{n_j\times n_j}$ such that $\Phi_j$ in (\ref{gs-26}) satisfies 
\begin{equation*}Q_{j}\Phi_{j}Q_{j}^T=\begin{bmatrix}
0 & D_{j}\\
-D_{j} & 0\\
\end{bmatrix},\ j=r-1,r,\end{equation*}
where $D_{j}=\mbox{diag}(\eta^{(j)}_1,\ldots,\eta^{(j)}_{(n_j/2)})$, $\eta_{k}>0$, $k=1,\ldots,\frac{n_j}{2}$. Let $D_{jj}=\mbox{diag}(D_{j}, D_{j})^{1/2}$. Then
\begin{equation}\label{xg-4-1}
D_{jj}^{-1}Q_j\Phi_jQ_j^TD_{jj}^{-1}=\begin{bmatrix}
0 & I_{(n_j/2)}\\
-I_{(n_j/2)} & 0\\
\end{bmatrix}, \ j=r-1,r.
\end{equation}
Moreover, 
 \begin{equation}\label{xg-2}
\hat{X}_j:=X_jQ^{-1}D_{jj},\ \ \hat{X}_jJ_j=X_jQ^{-1}D_{jj}J_j=X_j Q^{-1}J_jD_{jj}=X_jJ_jQ^{-1}D_{jj}, \ j=r-1,r.
\end{equation}

Let 
\begin{equation*}\hat{X}=[\hat{X}_1,\hat{X}_2,\ldots,\hat{X}_r],\ \ P=\mbox{diag}(Q_1,Q_2,\ldots,Q_{r}), \ \ D=\mbox{diag}(D_{11},D_{22},\ldots,D_{rr}).\end{equation*}
 From (\ref{xg-1}) and (\ref{xg-2}), we can obtain that 
\begin{equation}\label{gs-32}
\hat{X}=XP^{-1}D,\ \ \ \hat{X}J=XP^{-1}DJ=XP^{-1}JD=XJP^{-1}D,
\end{equation}
which implies that $(\hat{X},\ J)$ is a standard pair of $P(\lambda)$. Since all the non-real eigenvalues of $J$ are closed under the complex conjugate, it follows that the columns of $\hat{X}$ are still closed under the complex conjugate. From Lemma \ref{remark-1}, we can obtain that $(Y_R,\ J_R)$  is a real-valued standard pair of $P(\lambda)$, where $Y_R=\hat{X}M^{H}$ and $J_R=MJM^{H}$.

Since $(X_R,\ J_R)$ is a real standard pair of $P(\lambda)$, we can see from Theorem \ref{main-thm} that 
\begin{equation*}\Gamma_R=\left([X_R^T,\ J^{-T}_RX_R^T]\begin{bmatrix}
Q & A\\
A & 0\\
\end{bmatrix}\begin{bmatrix}
X_R\\
X_RJ_R\\
\end{bmatrix}\right)^{-1}.\end{equation*}
It follows from the proof of Theorem \ref{thm-gamma} that 
\begin{equation}\label{gs-33}
\Phi =M^{H}\Gamma_R(M^H)^T=\left([X^T,\ J^{-T}X^T]\begin{bmatrix}
Q & A\\
A & 0\\
\end{bmatrix}\begin{bmatrix}
X\\
XJ\\
\end{bmatrix}\right)^{-1}.
\end{equation}
Then, we can obtain from (\ref{xg-3}), (\ref{xg-4}), (\ref{xg-4-1}), (\ref{gs-32}) and (\ref{gs-33}) that 
\begin{equation}\label{xg-5}
\begin{array}{ll}
\tilde{\Gamma}_R &:= MD^{-1}P\Phi P^TD^{-1}M^T\\
&\ \ =M\left([\hat{X}^T,\ J^{-T}\hat{X}^T]\begin{bmatrix}
Q & A\\
A & 0\\
\end{bmatrix}\begin{bmatrix}
\hat{X}\\
\hat{X}J\\
\end{bmatrix}\right)^{-1}M^T\\
&\ \ =\left([Y_R^T,\ J_R^{-T}Y_R^T]\begin{bmatrix}
Q & A\\
A & 0\\
\end{bmatrix}\begin{bmatrix}
Y_R\\
Y_RJ_R\\
\end{bmatrix}\right)^{-1}\\
\end{array}
\end{equation}
 has the form as in (\ref{gs-31}), which implies that $\tilde{\Gamma}_R$ is indeed the matrix corresponding to $(Y_R,\ J_R)$.
\end{proof}
 
 \begin{theorem}\label{thm-s-2}
 Let $(X,\ J)$ defined by (\ref{gs-14-1}) be a standard pair of $(H, \epsilon)$-palindromic $P(\lambda)$ with $\mathbb{K}=\mathbb{C}$. Suppose that all the eigenvalues of $P(\lambda)$ are semi-simple, then there exist a standard pair $(\tilde{X},\ J)$ of $P(\lambda)$ such that the corresponding matrix $\Gamma\in\mathcal{S}_{(J,H,\epsilon)}$ can be expressed as
 \begin{equation}\label{gs-35}
 \Gamma=\mbox{diag}\left(\begin{bmatrix}
 0 & I_{n_1}\\
 -\epsilon I_{n_1} & 0\\
 \end{bmatrix},\cdots,\begin{bmatrix}
 0 & I_{n_{l}}\\
 -\epsilon I_{n_{l}} & 0\\
 \end{bmatrix}, \mathfrak{E}_{l+1}, \ldots, \mathfrak{E}_{r}\right),
 \end{equation}
 where $\mathfrak{E}_j=\mbox{diag}(\pm {\rm i}, \ldots, \pm {\rm i})$ is of order $n_j\times n_j$ when $\epsilon=1$ and  $\mathfrak{E}_j=\mbox{diag}(\pm 1, \ldots, \pm 1)$ is of order $n_j\times n_j$ when $\epsilon=-1$, $j=l+1,\ldots, r.$ 
 \end{theorem}
 
 \begin{proof}
 Since $(X,\ J)$ is a standard pair of $P(\lambda)$, it follows from Theorem \ref{thm-gamma-2} that $P(\lambda)$ has a spectral decomposition as in (\ref{gs-36}), in which the matrix $\Gamma\in\mathcal{S}_{(J,H,\epsilon)}$ has the form as in (\ref{gs-38}).

(1). $\lambda_j\in\mathbb{C}$ with $|\lambda_j|\neq 1$, $j=1,\ldots,l$. Let the SVD of $S_j$ be $S_j=Q_{j1}D_jQ_{j2}^H$, where $Q_{j1}$ and $Q_{j2}$ are unitary matrices and $D_j=\mbox{diag}(\sigma_1^{(j)}, \ldots, \sigma_{n_j}^{(j)})$, $\sigma_k^{(j)}>0$, $k=1,\ldots, n_j$. Define $D_{jj}=\mbox{diag}(D_j, D_j)^{1/2}$. Let $Q_j:=\mbox{diag}(Q_{j1},Q_{j2})$. Then, we have 
\begin{equation*}\tilde{S}_j:=D_{jj}^{-1}Q_j^H\begin{bmatrix}
0 & S_1\\
-\epsilon S_1^H & 0\\
\end{bmatrix}Q_jD_{jj}^{-1}=\begin{bmatrix}
0 & I_{n_j}\\
-\epsilon I_{n_j} & 0\\
\end{bmatrix}.\end{equation*}
Let $\tilde{X}_j:=X_jD_{jj}Q_j$. Then
$\tilde{X}_jJ_j=X_jD_{jj}Q_jJ_j=X_jD_{jj}J_jQ_j=X_jJ_jD_{jj}Q_j,\ j=1, \ldots, l,$
 since $N_j=0$ by assumption. 
 
(2). $\lambda_j\in\mathbb{C}$ with $|\lambda_j|=1$, $j=l+1,\ldots,r$. Clearly, there exists an unitary matrix $Q_j\in\mathbb{C}^{n_j\times n_j}$ such that $S_j=Q_j\Sigma_jQ_j^H$, where $\Sigma_j=\mbox{diag}(\delta_1^{(j)}, \ldots, \delta_{n_j}^{(j)})$ is nonsingular with $\delta_k^{(j)}\in{\rm i\mathbb{R}}$ for $H$-palindromic and $\delta_k^{(j)}\in\mathbb{R}$ for $H$-anti-palindromic, 
$k=1,\ldots, n_j$. Define $D_{jj}=|D_j|^{1/2}$. Then, we have 
$\tilde{S}_j:=D_{jj}^{-1}Q_j^HS_jQ_jD_{jj}^{-1}=E_j,$ where $E_j$ is defined in (\ref{gs-35}). Let $\tilde{X}_j:=X_jD_{jj}Q_j$, we have
$\tilde{X}_jJ_j=X_jD_{jj}Q_jJ_j=X_jD_jJ_jQ_j=X_jJ_jD_{jj}Q_j,\ j=l+1, \ldots, r.$

 Now, let 
 \begin{equation*}\tilde{X}=[\tilde{X}_1,\tilde{X}_2,\ldots, \tilde{X}_r],\ P=\mbox{diag}(Q_{1}, Q_2, \ldots, Q_r),\ D=\mbox{diag}(D_{11},D_{22},\ldots,D_{rr}).\end{equation*}
Note that $P$ is unitary and  
\begin{equation*}\tilde{\Gamma}:= D^{-1}P^H\Gamma PD^{-1}=\mbox{diag}(\tilde{S}_1,\ldots,\tilde{S}_r)\end{equation*}
 has the form as in (\ref{gs-35}). Moreover, we have $\tilde{X}=XPD$ and  $\tilde{X}J=XPDJ=XPJD=XJPD.$
Similar to the proof of (\ref{xg-5}), we can prove that $\tilde{\Gamma}$ is indeed the matrix corresponding to $(\tilde{X},\ J)$. 
 \end{proof}

\section{Solvability of PMUP}

Based on the spectral decomposition of the $(\star,\epsilon)$-structured palindromic $P(\lambda)$ and structures of $\mathcal{S}_{(\star,\epsilon)}$ provided in Section \ref{sec2} and \ref{sec3}, respectively, we show that the PMUP is solvable and  characterize the parameter solutions of PMUP. 
Partition the matrix pair $(X,\ J)\in\mathbb{C}^{n\times 2n}\times\in\mathbb{C}^{2n\times 2n}$ into
\begin{equation}\label{gg-1}
J=\mbox{diag}(\Lambda_1,\ \Lambda_2),\ \ X=[Y_1,\ Y_2],
\end{equation}  
where $(Y_1,\ \Lambda_1)\in\mathbb{C}^{n\times p}\times\in\mathbb{C}^{p\times p}$ is eigenpair that needs to be repaced, and $(Y_2,\ \Lambda_2)$ is eigenpair which needs to be kept unchanged. Let 
$
\tilde{Y}_1\in\mathbb{K}^{n\times p},\ \ \tilde{\Lambda}_1\in\mathbb{C}^{p\times p},
$
 where $\sigma(\tilde{\Lambda}_1)=\{\tilde{\lambda}_j\}_{j=1}^p$, and the $k$-column of $\tilde{Y}_1$ is  right eigenvector corresponding to the prescribed eigenvalue $\tilde{\lambda}_k$.  Similar to the proof of the Theorem 3.1 in \cite{Mao-laa}, we can get the following orthogonal relationship. 
\begin{lemma}\label{lem-4}
Suppose that $(X,\ J)$ is a standard pair of $P(\lambda)$ with partitions as in (\ref{gg-1}). If $\sigma(\Lambda_1)\cap\sigma(\Lambda_2^{-\star})=\emptyset$, then
\begin{equation}\label{gg-7}
[Y_2^{\star},\ \Lambda_2^{-\star}Y_2^{\star}]\begin{bmatrix}
Q & A\\
A & 0\\
\end{bmatrix}\begin{bmatrix}
Y_1\\
Y_1\Lambda_1\\
\end{bmatrix}=0.
\end{equation}
\end{lemma}

For the damped symmetric system, it has been shown in \cite{11_2007Updating} that a necessary solvable condition for the no spill-over MUP is that the matrices $Y_1$ and $\tilde{Y}_1$ should satisfy $\tilde{Y}_1=Y_1\Phi$ for some singular matrix $\Phi$.  Motivated by  \cite{11_2007Updating}, we give a sufficient solvable condition for the PMUP.

\begin{theorem}\label{thm-eep}
Suppose that $(X,\ J)$ is a standard pair of the $(\star,\epsilon)$-palindromic system $P(\lambda)$ with the partitions as in (\ref{gg-1}). If $\sigma(\Lambda_1)\cap\sigma(\Lambda_2^{-\star})=\emptyset$, then for any nonsingular matrix $\Phi\in\mathbb{K}^{p\times p}$ and $\tilde{\Gamma}_1\in\mathcal{S}_{(\tilde{\Lambda}_1,\star,\epsilon)}$ which satisfy  
\begin{equation}\label{gg-8}
\Phi\tilde{\Gamma}_1\Phi^{\star}=\Gamma_1,
\end{equation}
where
\begin{equation}\label{gg-9}
\Gamma_1=\left([Y_1^{\star},\ Y_1^{-\star}\Lambda_1^{\star}]\begin{bmatrix}
Q & A\\
A & 0\\
\end{bmatrix}\begin{bmatrix}
Y_1\\
Y_1\Lambda_1\\
\end{bmatrix}\right)^{-1},
\end{equation}
the matrices 
\begin{equation} \label{gg-10-1}
 \tilde{A}=\left(A^{-1}+(\tilde{Y}_1\tilde{\Lambda}_1\tilde{\Gamma}_1\tilde{Y}_1^{\star}-Y_1\Lambda_1\Gamma_1Y_1^{\star}) \right)^{-1},\ \  \tilde{Q}=\tilde{A}A^{-1}QA^{-1}\tilde{A}+\tilde{A}(Y_1\Lambda_1^2\Gamma_1Y_1^{\star}-\tilde{Y}_1\tilde{\Lambda}_1^2\tilde{\Gamma}_1\tilde{Y}_1^{\star})\tilde{A},
\end{equation}
form a solution of the PMUP.
\end{theorem}
\begin{proof}
Since $(X,\ J)$ is a standard pair of $P(\lambda)$, we can see from Theorem \ref{main-thm} that there is a nonsingular matrix $\Gamma\in\mathcal{S}_{(J,\star,\epsilon)}$ such that (\ref{gs-4}) and (\ref{gs-5}) are satisfied. Since 
$\sigma(\Lambda_1)\cap\sigma(\Lambda_2^{-\star})=\emptyset$, it follows from  Lemma \ref{lem-4} that $\Gamma=\mbox{diag}(\Gamma_1, \Gamma_2)$, where $\Gamma_1$ is given by (\ref{gg-9}). Clearly, $\Gamma_j\in\mathcal{S}_{(\Lambda_j,\star,\epsilon)}$, $j=1, 2$, and (\ref{gs-4}), (\ref{gs-5}) can be rewritten as 
\begin{equation}\label{gg-11}
\left\{\begin{array}{l}
0=Y_1\Gamma_1Y_1^{\star}+Y_2\Gamma_2Y_2^{\star},\\
 A^{-1}=Y_1\Lambda_1\Gamma_1Y_1^{\star}+Y_2\Lambda_2\Gamma_2Y_2^{\star}, \\ Q=-A(Y_1\Lambda_1^2\Gamma_1Y_1^{\star}+Y_2\Lambda_2^2\Gamma_2Y_2^{\star})A.\\
 \end{array}\right.
\end{equation}

From Theorem \ref{main-thm}, we know that if there exists a matrix  $\tilde{Y}_1\in\mathbb{K}^{n\times 2n}$ such that $\bigl[\begin{smallmatrix}
\tilde{Y}\\
\tilde{Y}\tilde{\Lambda}\\
\end{smallmatrix}\bigr]$ is nonsingular, where $\tilde{Y}=[\tilde{Y}_1, Y_2]$, $\tilde{J}=\mbox{diag}(\tilde{\Lambda}_1, \Lambda_2)$, and a nonsingular matrix $\tilde{\Gamma}_2\in\mathcal{S}_{(\Lambda_2,\star,\epsilon)}$ such that 
$\mbox{diag}(\tilde{\Gamma}_1, \tilde{\Gamma}_2)\in\mathcal{S}_{(\tilde{J},\star,\epsilon)}$ and 
\begin{equation}\label{gg-12}
\tilde{Y}_1\tilde{\Gamma}_1\tilde{Y}_1^{\star}+Y_2\tilde{\Gamma}_2\tilde{Y}_2^{\star}=0,
\end{equation}
then the matrices 
\begin{equation}\label{gg-13}
\tilde{A}=\left(\tilde{Y}_1\tilde{\Lambda}_1\tilde{\Gamma}_1\tilde{Y}_1^{\star}+Y_2\Lambda_2\tilde{\Gamma}_2Y_2^{\star}\right)^{-1},\ \ 
\tilde{Q}=-\tilde{A}(\tilde{Y}_1\tilde{\Lambda}_1^2\tilde{\Gamma}_1\tilde{Y}_1^{\star}+Y_2\Lambda_2^2\tilde{\Gamma}_2Y_2^{\star})\tilde{A},
\end{equation}
form a solution of PMUP. We set $\tilde{Y}_1=Y_1\Phi$ and $\tilde{\Gamma}_2=\Gamma_2$. Then, it follows from (\ref{gg-8}) and (\ref{gg-11}) that (\ref{gg-12}) is satisfied. Substituting (\ref{gg-11}) into 
(\ref{gg-13}), we can obtain that the solution can be given by (\ref{gg-10-1}).
\end{proof}

It can be seen that Theorem \ref{thm-eep} imposes no restrictions on the algebraic and geometric multiplicities of eigenvalues of $\Lambda_1$ and $\tilde{\Lambda}_1$. And the solutions given by (\ref{gg-10-1}) do not need any infromation of $(Y_2,\ \Lambda_2)$, which are generally unknown in practice.  Next, we will characterize the solutions of PMUP for both simple and multiple eigenvalues. 

\subsection{Solutions of PMUP with simple eigenvalues}

In this subsection, we always assume that $\{\lambda_j\}_{j=1}^p$ and $\{\tilde{\lambda}_j\}_{j=1}^p$ are simple, and $|\lambda_j|\neq 1$, $|\tilde{\lambda}_j|\neq 1$, $j=1,\ldots, p$. First, we give the form of $\Gamma_1$  and the matrix $\tilde{\Gamma}_1$ is always chosen in the simpler form. 

Case 1. $(T, \epsilon)$-palindromic with $\mathbb{K}=\mathbb{R}$:
$$
\Lambda_{1}=\mbox{diag}\left(\begin{bmatrix}
\alpha_1 & \beta_1\\
-\beta_1 & \alpha_1\\
\end{bmatrix},\begin{bmatrix}
\alpha'_1 & \beta'_1\\
-\beta'_1 & \alpha'_1\\
\end{bmatrix},\cdots,\begin{bmatrix}
\alpha_{k_1} & \beta_{k_1}\\
-\beta_{k_1} & \alpha_{k_1}\\
\end{bmatrix},\begin{bmatrix}
\alpha'_{k_1} & \beta'_{k_1}\\
-\beta'_{k_1} & \alpha'_{k_1}\\
\end{bmatrix}, \lambda_{1}, \frac{1}{\lambda_{1}}, \ldots, \lambda_{k_2}, \frac{1}{\lambda_{k_2}} \right)\in\mathbb{R}^{p\times p},$$
$$
\tilde{\Lambda}_{1}=\mbox{diag}\left(\begin{bmatrix}
\tilde{\alpha}_1 & \tilde{\beta}_1\\
-\tilde{\beta}_1 & \tilde{\alpha}_1\\
\end{bmatrix},\begin{bmatrix}
\tilde{\alpha}'_1 & \tilde{\beta}'_1\\
-\tilde{\beta}'_1 & \tilde{\alpha}'_1\\
\end{bmatrix},\cdots,\begin{bmatrix}
\tilde{\alpha}_{\tilde{k}_1} & \tilde{\beta}_{\tilde{k}_1}\\
-\tilde{\beta}_{\tilde{k}_1} & \tilde{\alpha}_{\tilde{k}_1}\\
\end{bmatrix},\begin{bmatrix}
\tilde{\alpha}'_{\tilde{k}_1} & \tilde{\beta}'_{\tilde{k}_1}\\
-\tilde{\beta}'_{k_1} & \tilde{\alpha}'_{\tilde{k}_1}\\
\end{bmatrix}, \tilde{\lambda}_{1}, \frac{1}{\tilde{\lambda}_{1}}, \ldots, \tilde{\lambda}_{\tilde{k}_2}, \frac{1}{\tilde{\lambda}_{\tilde{k}_2}} \right)\in\mathbb{R}^{p\times p},$$
 where $\alpha_j'+i\beta'_j=\frac{1}{\alpha_j+i\beta_j}$, $\tilde{\alpha}'_l+i\tilde{\beta}'_l=\frac{1}{\tilde{\alpha}_l+i\tilde{\beta}_l}$, $j=1,\ldots,k_1$, $l=1,\ldots,\tilde{k}_1$, and $4k_1+2k_2=4\tilde{k}_1+2\tilde{k}_2=p$. 
 By straightforward  calculation, we can obtain that the matrix $\Gamma_{1}\in\mathcal{S}_{(\Lambda_1,T,\epsilon)}$ has the following form   
\begin{equation*}\Gamma_{1}=\mbox{diag}\left(\begin{bmatrix}
0 & U_1\\
-\epsilon U_1 & 0\\
\end{bmatrix},\cdots, \begin{bmatrix}
0 & U_{k_1}\\
-\epsilon U_{k_1} & 0\\
\end{bmatrix}, \begin{bmatrix}
0 & \xi_{1}\\
-\epsilon \xi_{1} & 0\\
\end{bmatrix},\cdots, \begin{bmatrix}
0 & \xi_{k_2}\\
-\epsilon \xi_{k_2} & 0\\
\end{bmatrix}\right),
\end{equation*}
where $U_j=\bigl[\begin{smallmatrix}
a_j & b_j\\
b_j & -a_j\\
\end{smallmatrix}\bigr]\in\mathbb{R}^{2\times 2}$ and $a_j,b_j,\xi_j\in\mathbb{R}$ are nonzero. By Algorithm 4.1 in \cite{Cai_2010}, we can see that the matrix $U_j$ is congruent to $\bigl[\begin{smallmatrix}
1 & 0\\
0 & -1\\
\end{smallmatrix}\bigr]$. Choose $\tilde{\Gamma}_{1}$ as 
\begin{equation}\label{xgx-2}
\tilde{\Gamma}_{1}=\mbox{diag}\left(\begin{bmatrix}
0 & \tilde{U}_1\\
-\epsilon \tilde{U}_1 & 0\\
\end{bmatrix},\cdots, \begin{bmatrix}
0 & \tilde{U}_{\tilde{k}_1}\\
-\epsilon \tilde{U}_{\tilde{k}_1} & 0\\
\end{bmatrix}, \begin{bmatrix}
0 & 1\\
-\epsilon  & 0\\
\end{bmatrix},\cdots, \begin{bmatrix}
0 & 1\\
-\epsilon  & 0\\
\end{bmatrix}\right)\in\mathbb{R}^{p\times p},
\end{equation}
where $\tilde{U}_l=\bigl[\begin{smallmatrix}
1 & 0\\
0 & -1\\
\end{smallmatrix}\bigr]$, $l=1,\ldots,\tilde{k}_1$. It is easy to verify that $\tilde{\Gamma}_{1}\in\mathcal{S}_{(\tilde{\Lambda}_{1},T,\epsilon)}$, and there exist two nonsingular matrices $P_1, P_2\in\mathbb{R}^{p\times p}$ such that 
\begin{equation}\label{gg-14}
P_1\tilde{\Gamma}_{1}P_1^T=P_2\Gamma_{1}P_2^T=\begin{bmatrix}
0 & I_q\\
-\epsilon I_q & 0\\
\end{bmatrix}:=\mathcal{R}.
\end{equation}

Case 2. $(H, \epsilon)$-palindromic with $\mathbb{K}=\mathbb{R}$:
\begin{equation*}\Lambda_1=\mbox{diag}\left(\lambda_1,\frac{1}{\bar{\lambda}_1},\lambda_2,\frac{1}{\bar{\lambda}_2}\ldots,\lambda_{q},\frac{1}{\bar{\lambda}_q}\right)\in\mathbb{C}^{p\times p},\end{equation*} \begin{equation*}\tilde{\Lambda}_1=\mbox{diag}\left(\tilde{\lambda}_1,\frac{1}{\bar{\tilde{\lambda}}_1},\tilde{\lambda}_2,\frac{1}{\bar{\tilde{\lambda}}_2},\ldots,\tilde{\lambda}_{q},\frac{1}{\bar{\tilde{\lambda}}_q}\right)\in\mathbb{C}^{p\times p}.\end{equation*}
Then, we have 
\begin{equation*}\Gamma_1=\mbox{diag}\left(\begin{bmatrix}
0 & \eta_1\\
-\epsilon \eta_1 & 0\\
\end{bmatrix},\begin{bmatrix}
0 & \eta_2\\
-\epsilon \eta_2 & 0\\
\end{bmatrix}, \cdots, \begin{bmatrix}
0 & \eta_{q}\\
-\epsilon \eta_{q} & 0\\
\end{bmatrix}\right)\in\mathbb{C}^{p\times p},\end{equation*}
where $\eta_1,\ldots, \eta_q$ are nonzero. Choose $\tilde{\Gamma}_1$ as 
\begin{equation}\label{xgx-3}
\tilde{\Gamma}_1=\mbox{diag}\left(\begin{bmatrix}
0 & 1\\
-\epsilon & 0\\
\end{bmatrix},\begin{bmatrix}
0 & 1\\
-\epsilon & 0\\
\end{bmatrix}, \cdots, \begin{bmatrix}
0 & 1\\
-\epsilon & 0\\
\end{bmatrix}\right)\in\mathbb{C}^{p\times p}.
\end{equation}
Clearly, $\tilde{\Gamma}_1\in\mathcal{S}_{(\tilde{\Lambda}_1,H,\epsilon)}$, and there exist two nonsingular matrices $P_3, P_4\in\mathbb{C}^{p\times p}$ such that 
\begin{equation}\label{gg-14-1}
P_3\tilde{\Gamma}_1P_3^H=P_4\Gamma_1P_4^H=\begin{bmatrix}
0 & I_q\\
-\epsilon I_q & 0\\
\end{bmatrix}:=\mathcal{R}.
\end{equation}

Substituting (\ref{gg-14}) and (\ref{gg-14-1}) into (\ref{gg-8}) leads to 
\begin{equation}\label{gg-15}
\widetilde{\Phi}\mathcal{R}\widetilde{\Phi}^{\star}=\mathcal{R},
\end{equation}
where $\widetilde{\Phi}=P_{2t}\Phi P_{2t-1}^{-1}$, $t=1, 2.$ By straightforward calculation, the nonsingular solution $\widetilde{\Phi}$ of (\ref{gg-15}) can be characterized by the following lemma.
\begin{lemma}\label{lem-5}
 The solution of (\ref{gg-15})  can be given by  
\begin{equation}\label{gg-16-x}
\widetilde{\Phi}=\begin{bmatrix}
\Phi_{1} & \Phi_{2}\\
-\epsilon\Phi_2 & \Phi_1\\
\end{bmatrix},
\end{equation}
where $\Phi_1, \Phi_2$ are arbitrary $q\times q$ matrices which satisfy $\widetilde{\Phi}\widetilde{\Phi}^{\star}=I_p$. 
\end{lemma}

\begin{algorithm}[h]
\caption{ Finding solutions of PMUP with simple eigenvalues.} 
\label{alg-eep} 
\begin{algorithmic} 
\STATE {\bf Input}: $(\star,\epsilon)$-palindromic polynomial $\lambda^2A+\lambda Q+\epsilon A^{\star}$, $\{x_j, \lambda_j\}_{j=1}^p$ and $\{\tilde{\lambda}_j\}_{j=1}^p$.
\STATE {\bf Output} The matrices $\tilde{A}$, $\tilde{Q}$.
\STATE \textbf{1.} Form the matrices $Y_{1}, \Lambda_{1}, \tilde{\Lambda}_{1}$ and compute $\Gamma_{1}$ by (\ref{gg-9}).
\STATE \textbf{2.} According to the values of $\star$ and $\epsilon$, choose the matrix $\tilde{\Gamma}_{1}$ as (\ref{xgx-2}) or (\ref{xgx-3}).
\STATE \textbf{3.} Computing the matrices $P_1, P_2$ by (\ref{gg-14}) or $P_3, P_4$ by (\ref{gg-14-1}). 
\STATE \textbf{4.} Randomly choose $\widetilde{\Phi}\in\mathbb{K}^{p\times p}$ by (\ref{gg-16-x}), and compute $\Phi=P_{2}\widetilde{\Phi}P_{1}^{-1}$ or  $\Phi=P_{4}\widetilde{\Phi}P_{3}^{-1}$.
\STATE \textbf{5.} Set $\tilde{Y}_{1}=Y_{1}\Phi$ and compute $\tilde{A}$ and $\tilde{Q}$ by (\ref{gg-10-1}).
\end{algorithmic}
\end{algorithm}

\subsection{Solutions of PMUP with multiple eigenvalues}
 
In this subsection, we only consider the PMUP for the  $(H,\epsilon)$-palindromic system, and the case of $(T,\epsilon)$-palindromic can be discussed in a manner similar way. Without loss of generality, we assume that
all the distinct eigenvalues of $\{\lambda_j\}_{j=1}^p$ are $\lambda_1,\frac{1}{\bar{\lambda}_1},\ldots,\lambda_k,\frac{1}{\bar{\lambda}_k},\lambda_{2k+1},\ldots,\lambda_q$ with $|\lambda_j|\neq 1$ for $j=1,\ldots, k$, $|\lambda_j|=1$ for $j=2k+1,\ldots,q$, and  the algebraic multiplicity of $\lambda_j$ is $n_j$, $j=1,\ldots, q$. Let 
\begin{align}
&\Lambda_1=\mbox{diag}\left(J^{(1)}(\lambda_1),J^{(1)}(1/\bar{\lambda}_1)\ldots,J^{(1)}(\lambda_k),J^{(1)}(1/\bar{\lambda}_k),J^{(1)}(\lambda_{2k+1}),\ldots,J^{(1)}(\lambda_q)\right),\label{xg-7-1}\\
& Y_1=\left[Y^{(1)}_1,Y^{(1)}_2,\ldots,Y^{(1)}_q\right],\label{xgx-7-2}
\end{align}
where $J^{(1)}(\lambda_j)=\lambda_jI_{n_j}+N_j\in\mathbb{C}^{n_j\times n_j}$ is the complex Jordan canonical form of $\lambda_j$ and the columns of $Y^{(1)}_j\in\mathbb{C}^{n\times n_j}$ are corresponding generalized eigenvectors.
Suppose that all the distinct eigenvalues of $\{\tilde{\lambda}_j\}_{j=1}^p$ are $\mu_1,\frac{1}{\bar{\mu}_1},\ldots,\mu_{\tilde{k}},\frac{1}{\bar{\mu}_{\tilde{k}}}$, $\mu_{2\tilde{k}+1},\ldots,\mu_q$ with 
$|\mu_j|\neq 1$ for $j=1,\ldots, \tilde{k}$, $|\mu_j|=1$ for $j=2\tilde{k}+1,\ldots,q$, and the algebraic multiplicity of $\mu_j$ is ${\tilde{n}}_j$, $j=1,\ldots,q$. Let
\begin{equation}\label{xgx-8}
\tilde{\Lambda}_1=\mbox{diag}\left(J^{(1)}(\mu_1),J^{(1)}(1/\bar{\mu}_1)\ldots,J^{(1)}(\mu_{\tilde{k}}),J^{(1)}(1/\bar{\mu}_{\tilde{k}}),J^{(1)}(\mu_{2{\tilde{k}}+1}),\ldots,J^{(1)}(\mu_q)\right),
\end{equation}
where $J^{(1)}(\mu_j)=\mu_jI_{\tilde{n}_j}+N_j\in\mathbb{C}^{{\tilde{n}}_j\times {\tilde{n}}_j}$ is the complex Jordan canonical form of $\mu_j$. In this section, we always assume that $\sum_{j=1}^kn_j\geq \sum_{j=1}^{\tilde{k}}\tilde{n}_j$. 

Since $\sigma(\Lambda_1)\cap \sigma(\Lambda_2^{-H})=\emptyset$, we can obtain from Lemma \ref{lem-4} and Theorem \ref{thm-gamma-2} that the matrix $\Gamma_1$ defined by (\ref{gg-9}) satisfies $\Gamma_1\in\mathcal{S}_{(\Lambda_1,H,\epsilon)}$ and has the following form 
\begin{equation}\label{xgx-9}
\begin{array}{ll}
\Gamma_1=&\mbox{diag}\left(\begin{bmatrix}
0 & S_1\\
-\epsilon S_1^H & 0\\
\end{bmatrix},\cdots,\begin{bmatrix}
0 & S_{k}\\
-\epsilon S_{k}^H & 0\\
\end{bmatrix},S_{2k+1},\ldots,S_q\right),\\
\end{array}
\end{equation}
where $S_j\in\Theta_{(\lambda_j, 1/\lambda_j, j)}$ for $j=1,\ldots, k$ and $S_{j}\in\Theta_{(\lambda_j, \bar{\lambda}_j, j)}$ with $S_{j}^H=-\epsilon S_{j}$ for $j=2k+1,\ldots, q$.  Randomly choose $\tilde{\Gamma}_1$ as 
\begin{equation}\label{xgx-9-1}
\begin{array}{ll}
\tilde{\Gamma}_1=&\mbox{diag}\left(\begin{bmatrix}
0 & \tilde{S}_1\\
-\epsilon \tilde{S}_1^H & 0\\
\end{bmatrix},\cdots,\begin{bmatrix}
0 & \tilde{S}_{\tilde{k}}\\
-\epsilon \tilde{S}_{\tilde{k}}^H & 0\\
\end{bmatrix},\tilde{S}_{2\tilde{k}+1},\ldots,\tilde{S}_q\right),\\
\end{array}
\end{equation}
where $\tilde{S}_j\in\Theta_{(\mu_j, 1/\mu_j, j)}$ for $j=1,\ldots, \tilde{k}$ and $\tilde{S}_{j}\in\Theta_{(\mu_j, \bar{\mu}_j, j)}$ with $\tilde{S}_{j}^H=-\epsilon \tilde{S}_{j}$ for $j=2\tilde{k}+1,\ldots, q$. From  Theorem \ref{thm-gamma-2} and the definition of $\Theta_{(a,b,j)}$ in (\ref{gs-19}), it is easy to verify that the matrix $\tilde{\Gamma}_1$ given by (\ref{xgx-9-1}) satisfies $\tilde{\Gamma}_1\in\mathcal{S}_{(\tilde{\Lambda}_1,H,\epsilon)}$. 

Similar to the proof of Theorem \ref{thm-s-2}, there exist a nonsingular matrix $\tilde{P}_1\in\mathbb{C}^{q\times q}$ such that 
\begin{align}
&\tilde{P}_1\Gamma_1\tilde{P}_1^H=\mbox{diag}\left(\begin{bmatrix}
 0 & I_{n_1}\\
 -\epsilon I_{n_1} & 0\\
 \end{bmatrix},\cdots,\begin{bmatrix}
 0 & I_{n_{k}}\\
 -\epsilon I_{n_{k}} & 0\\
 \end{bmatrix}, \mathfrak{E}_{2k+1}, \ldots, \mathfrak{E}_{q}\right),\label{xgx-10-1}
\end{align}
where $\mathfrak{E}_j=\mbox{diag}(\pm {\rm i}, \ldots, \pm {\rm i})\in\mathbb{C}^{n_j\times n_j}$ when $\epsilon=1$ and $\mathfrak{E}_j=\mbox{diag}(\pm 1, \ldots, \pm 1)\in\mathbb{C}^{n_j\times n_j}$ when $\epsilon=-1$. Let $m=\sum_{t=1}^kn_t-\sum_{t=1}^{\tilde{k}}\tilde{n}_{t}$.
As is known, the matrix $\bigr[\begin{smallmatrix}
  0 & I_m\\
  -\epsilon I_m & 0\\
  \end{smallmatrix}\bigr]$ is unitary similar to $\mathfrak{N}$, where $\mathfrak{N}=\mbox{diag}({\rm i}I_m, -{\rm i}I_m)$ when $\epsilon=1$ and $\mathfrak{N}=\mbox{diag}(I_m, -I_m)$ when $\epsilon=-1$. Therefore, we can always choose suitable $\tilde{\Gamma}_1$ of form (\ref{xgx-9-1}) which satisfies 
  \begin{align}
&\tilde{P}_2\tilde{\Gamma}_1\tilde{P}_2^H=\mbox{diag}\left(\begin{bmatrix}
 0 & I_{\tilde{n}_1}\\
 -\epsilon I_{\tilde{n}_1} & 0\\
 \end{bmatrix},\cdots,\begin{bmatrix}
 0 & I_{\tilde{n}_{\tilde{k}}}\\
 -\epsilon I_{\tilde{n}_{\tilde{k}}} & 0\\
 \end{bmatrix},\begin{bmatrix}
 0 & I_{m}\\
 -\epsilon I_{m} & 0\\
 \end{bmatrix}, \mathfrak{E}_{2k+1}, \ldots, \mathfrak{E}_{q}\right),\label{xgx-10-2}
\end{align}
where $\tilde{P}_2\in\mathbb{C}^{q\times q}$ is nonsingular.
Then, we can see from (\ref{xgx-10-1}) and (\ref{xgx-10-2}) that there exist two permutation matrices $\tilde{Q}_1,\tilde{Q}_2\in\mathbb{C}^{q\times q}$ such that 
 \begin{equation}\label{xgx-11}
 \tilde{Q}_1\tilde{P}_1\Gamma_1\tilde{P}_1^H\tilde{Q}^H_1=\tilde{Q}_2\tilde{P}_2\tilde{\Gamma}_1\tilde{P}_2^H\tilde{Q}^H_2=
 \mbox{diag}\left(\begin{bmatrix}
 0 & I_{s}\\
 -\epsilon I_{s} & 0\\
 \end{bmatrix},\mathfrak{D}\right):=\mathcal{K},
 \end{equation}
 where $s=\sum_{t=1}^kn_t$ and $\mathfrak{D}=\mbox{diag}(\mathfrak{E}_{2k+1}, \ldots, \mathfrak{E}_{q})$ is of order $(q-2s)\times (q-2s)$.
 
 Substituting (\ref{xgx-11})  into (\ref{gg-8}) leads to 
\begin{equation}\label{gg-15}
\widetilde{\Phi}\mathcal{K}\widetilde{\Phi}^{H}=\mathcal{K},
\end{equation}
where $\widetilde{\Phi}=\tilde{Q}_1\tilde{P}_1\Phi\tilde{P}_2^{-1}\tilde{Q}_2^{-1}$. By Lemma \ref{lem-5}, the solution of (\ref{gg-15}) can be given by  
\begin{equation}\label{gg-16-xx}
\widetilde{\Phi}=\mbox{diag}\left(\begin{bmatrix}
\Phi_{1} & \Phi_{2}\\
-\epsilon\Phi_2 & \Phi_1\\
\end{bmatrix},I_{q-2s}\right),
\end{equation}
where $\Phi_1, \Phi_2\in\mathbb{C}^{s\times s}$ are arbitrary which satisfy $\widetilde{\Phi}\widetilde{\Phi}^{H}=I_q$. 
We summarize the discussion above in the following Algorithm \ref{alg-eep-1}.
\begin{algorithm}[h]
\caption{ Finding solutions of PMUP with multiple eigenvalues.} 
\label{alg-eep-1} 
\begin{algorithmic} 
\STATE {\bf Input}: $(H,\epsilon)$-palindromic polynomial $\lambda^2A+\lambda Q+\epsilon A^{H}$, $\{x_j, \lambda_j\}_{j=1}^p$ and $\{\tilde{\lambda}_j\}_{j=1}^p$.
\STATE {\bf Output} The matrices $\tilde{A}$, $\tilde{Q}$.
\STATE \textbf{1.} Form the matrices $Y_{1}, \Lambda_{1}, \tilde{\Lambda}_{1}$ and compute $\Gamma_{1}$ by (\ref{gg-9}).
\STATE \textbf{2.} Randomly choose the matrix $\tilde{\Gamma}_{1}$ as (\ref{xgx-9-1}) and compute $m=\sum_{t=1}^kn_t-\sum_{t=1}^{\tilde{k}}\tilde{n}_{t}$.
\STATE \textbf{3.} Computing the matrices $\tilde{Q}_1, \tilde{Q}_2, \tilde{P}_1, \tilde{P}_2$ by (\ref{xgx-10-1}), (\ref{xgx-10-2}) and (\ref{xgx-11}), respectively. 
\STATE \textbf{4.} Randomly choose $\widetilde{\Phi}\in\mathbb{C}^{q\times q}$ by (\ref{gg-16-xx}), and compute $\Phi=\tilde{P}_1^{-1}\tilde{Q}_1\widetilde{\Phi}\tilde{Q}_2\tilde{P}_2$.
\STATE \textbf{5.} Set $\tilde{Y}_{1}=Y_{1}\Phi$ and compute $\tilde{A}$ and $\tilde{Q}$ by (\ref{gg-10-1}).
\end{algorithmic}
\end{algorithm}

\begin{remark}
The assumption of $\sum_{j=1}^kn_j\geq \sum_{j=1}^{\tilde{k}}\tilde{n}_j$ in this subsection is reasonable, since the model one eigenvalues of the original system can not be replaced by the eigenvalues which are not model one. In fact, If $\sum_{j=1}^kn_j< \sum_{j=1}^{\tilde{k}}\tilde{n}_j$. Suppose that $\lambda_1,\lambda_2$ are two eigenvalues of $(H,1)$-palindromic system with $|\lambda_1|=|\lambda_2|=1$, which need to be replaced, and the corresponding matrix $\Gamma_1$ of $\mbox{diag}(\lambda_1, \lambda_2)$ is $\mbox{diag}({\rm i},\ {\rm i})$. It is easy to verify that $(\lambda_1,\lambda_2)$ can not be replaced by $(\mu, \frac{1}{\bar{\mu}})$ with $|\mu|\neq 1$, since the corresponding matrix $\tilde{\Gamma}_1=\begin{bmatrix}
0 & a\\
-\bar{a} & 0\\
\end{bmatrix}$ of $(\mu, \frac{1}{\bar{\mu}})$ can not be congruent to $\Gamma_1$. If $\sum_{j=1}^kn_j\geq \sum_{j=1}^{\tilde{k}}\tilde{n}_j$, similar to the discussion of (\ref{xgx-10-2}), we can always choose suitable $\tilde{\Gamma}_1$ for the given $\Gamma_1$ such that the sufficient solvable condition (\ref{gg-8}) of Theorem \ref{thm-eep} is satisfied. 

\end{remark}

\section{Numerical examples}

\begin{example}\label{exa-eep1}
In this example, we consider the PMUP of $T$-palindromic system $P(\lambda)=\lambda^2A+\lambda Q+A^T$ with 
\begin{equation*}A=\begin{bmatrix}
0.8147 &   0.6324 &   0.9575  &  0.9572\\
    0.9058 &   0.0975 &   0.9649  &  0.4854\\
    0.1270 &   0.2785 &   0.1576  &  0.8003\\
    0.9134 &   0.5469 &   0.9706  &  0.1419\\
    \end{bmatrix},\ \ Q=\begin{bmatrix}
1.8435 &   1.5715   & 1.4709 &   1.6150\\
    1.5715 &  -0.0714 &   1.6069 &   1.1052\\
    1.4709 &   1.6069 &   1.4863 &   1.0983\\
    1.6150 &   1.1052 &   1.0983 &   0.0637\\
    \end{bmatrix},\end{equation*}
    which are randomly generated. The eigenvalues of $P(\lambda)$ are $\{-1.1492\pm 0.5941i$, $-0.6866\pm 0.3550i$,
    $-4.1054$, $1.9390$, $0.5157$, $-0.2436\}$. Suppose that all the real eigenvalues are replaced by $\left\{-1+2i,-1-2i,\frac{1}{-1+2i},\frac{1}{-1-2i}\right\}$. 
\end{example}
By the given information, we can get $\Gamma_{1}=\mbox{diag}\left(\bigl[\begin{smallmatrix}
0 & 3.0441\\
-3.0441 & 0\\
\end{smallmatrix}\bigr], \bigl[\begin{smallmatrix}
0 & 14.8606\\
-14.8606 & 0\\
\end{smallmatrix}\bigr]\right).$
     Choosing $\widetilde{\Phi}=I_4$ and $\tilde{\Gamma}_{1}=\bigl[\begin{smallmatrix}
0 & U\\
-U & 0\\
\end{smallmatrix}\bigr]$ with $U=\mbox{diag}(1,-1)$, we can obtain from Algorithm \ref{alg-eep} that \begin{equation*}Y_{1}=\begin{bmatrix}
-0.5686 &  -0.5385 &   0.6906 &  -0.6729\\
   -0.5097 &  -0.1544 &  -0.1127 &   0.1516\\
    0.6456 &   0.4735 &  -0.7136 &   0.7146\\
   -0.0028 &   0.6797 &   0.0342 &   0.1165\\
   \end{bmatrix},\ \tilde{Y}_{1}=\begin{bmatrix}
1.6391 &  -9.9993 &  -0.5686  &  0.6906\\
    0.4700 &   2.2530 &  -0.5097 &  -0.1127\\
   -1.4414 &  10.6196 &   0.6456 &  -0.7136\\
   -2.0692 &   1.7319 &  -0.0028 &   0.0342\\
\end{bmatrix},\end{equation*}
and
\begin{equation*}\tilde{A}=\begin{bmatrix}
0.2835 &   0.3007 &   0.0808  &  0.7761\\
    0.2886 &   0.2748 &   0.2803 &  -0.1003\\
    0.3205 &   0.3777 &   0.0966 &   0.6529\\
    0.6080 &  -0.1364 &   0.4769 &   0.5826\\
    \end{bmatrix},\ \ \tilde{Q}=\begin{bmatrix}
     0.7773 &   0.4984&    0.6229  &  1.3387\\
    0.4984  &  0.5087 &   0.4176   &-0.2612\\
    0.6229  &  0.4176 &   0.4187   & 1.2751\\
    1.3387 &  -0.2612 &   1.2751   & 0.9606\\
    \end{bmatrix},\end{equation*}
which satisfy  
\begin{equation*} ||\tilde{A}\tilde{Y}_{1}\tilde{\Lambda}_{1}^2+\tilde{Q}\tilde{Y}_{1}\tilde{\Lambda}_{1}-\tilde{A}^T\tilde{Y}_{1}||_F=2.1331e-13, \ \  ||\tilde{A}Y_2{\Lambda}_2^2+\tilde{Q}Y_2{\Lambda}_2-\tilde{A}^TY_2||_F=3.7007e-14.
\end{equation*}

\begin{example}\label{exa-eep2}
Consider the PMUP of the $T$-anti-palindromic system $P(\lambda)=\lambda^2A+\lambda Q-A^T$ with 
\begin{equation*}A=\begin{bmatrix}
   1.4218  &  0.6557 &   0.6787  &  0.6555\\
    0.9157 &   1.0357&    0.7577 &   0.1712\\
    0.7922 &   0.8491&    0.7431 &   0.7060\\
    0.9595 &   0.9340&    0.3922 &   1.0318\\
\end{bmatrix}, Q=\begin{bmatrix}
       0   &-1.2734  &  0.8305 &   3.0438\\
    1.2734 &        0&    1.6864&   -2.0615\\
   -0.8305 &  -1.6864&         0&   -1.1703\\
   -3.0438 &   2.0615&    1.1703&         0\\
   \end{bmatrix}.\end{equation*}
   The eigenvalues of $P(\lambda)$ are $\{-3.4598\pm 4.2250i, -0.1150\pm 0.1415i, 1.2894\pm 2.5693i, 0.1560\pm 0.3109i\}$. Suppose that the eigenvalues $\{1.2894\pm 2.5693i, 0.1560\pm 0.3109i\}$ are replaced by $\left\{-1+2i,-1-2i,\frac{1}{-1+2i},\frac{1}{-1-2i}\right\}$.
\end{example}

By the given information, we can get $\Gamma_1=\bigl[\begin{smallmatrix}
0  & V\\
V & 0\\
\end{smallmatrix}\bigr]$, where $V=\bigl[\begin{smallmatrix}
0.1262 & 1.2571\\
1.2571 & -0.1262\\
\end{smallmatrix}\bigr]$.  Taking $\widetilde{\Phi}=I_4$ and $\tilde{\Gamma}_{1}=\bigl[\begin{smallmatrix}
0 & U\\
U & 0\\
\end{smallmatrix}\bigr]$ with $U=\mbox{diag}(1,-1)$, we can obtain from Algorithm \ref{alg-eep} that 
\begin{equation*}Y_{1}=\begin{bmatrix}
 -0.5446  &  0.0824  & -0.4964&   -0.0915\\
   -0.4204&    0.2996&   -0.4058&    0.1124\\
    1.1766&   -0.0364&    1.1796&   -0.3096\\
    0.0130&    0.2104&   -0.0728&   -0.2744\\
\end{bmatrix},\ \ \tilde{Y}_{1}=\begin{bmatrix}
 0.3918  & -0.4794   & 0.4827&   -0.2980\\
    0.1245&   -0.5667&    0.2535&   -0.3997\\
   -0.9533&    0.9175&   -0.7498&    1.1475\\
   -0.1694&   -0.1655&    0.2676&    0.1738\\
    \end{bmatrix},\end{equation*}
    and 
\begin{equation*}\tilde{A}=\begin{bmatrix}
  1.1158  &  1.0657  &  1.0998 &   1.3254\\
    0.3108&    1.0519&    0.8738 &   0.3189\\
   -0.0363&    0.6257&    0.6850 &   0.5383\\
    0.6483&    1.6292&    1.0696 &   2.1548\\
   \end{bmatrix},\ \ \tilde{Q}=\begin{bmatrix}
     0.0000 &  -0.8299&    1.2312&    3.2183\\
    0.8299 &   0.0000&    1.6988 &  -2.4279\\
   -1.2312 &  -1.6988&    0.0000 &  -1.4642\\
   -3.2183 &   2.4279&    1.4642 &  -0.0000\\
   \end{bmatrix},\end{equation*}
which satisfy
\begin{equation*}
||\tilde{A}\tilde{Y}_{1}\tilde{\Lambda}_{1}^2+\tilde{Q}\tilde{Y}_{1}\tilde{\Lambda}_{1}-\tilde{A}^T\tilde{Y}_{1}||_F=1.8039e-14,\ \ ||\tilde{A}Y_2{\Lambda}_2^2+\tilde{Q}Y_2{\Lambda}_2-\tilde{A}^TY_2||_F=4.3117e-14.
\end{equation*}

Numerical results of Examples \ref{exa-eep1} and \ref{exa-eep2}  show that all the prescribed eigenvalues can be reproduced accurately by the updated system, and the remaining eigenvalues and their associated eigenvectors are kept unchanged. Moreover, for the $(T, \epsilon)$-palindromic system, the number of eigenvalues of $\Lambda_1$ which occur in quadruples needs not be equal to the number of those eigenvalues of  $\tilde{\Lambda}_1$.

\begin{example}\label{exa-eep3}
Consider the PMUP of randomly chosen $H$-palindromic system $P(\lambda)=\lambda^2A+\lambda Q+A^H$ with 
{\scriptsize$$A=\begin{bmatrix}
0.8147 + 0.7547i &  0.5469 + 0.9597i &  0.8003 + 0.6991i &  0.0357 + 0.8407i &  0.6555 + 0.2511i  & 0.8235 + 0.9172i &  0.7655 + 0.0540i\\
   0.9058 + 0.2760i &  0.9575 + 0.3404i &  0.1419 + 0.8909i &  0.8491 + 0.2543i &  0.1712 + 0.6160i &  0.6948 + 0.2858i &  0.7952 + 0.5308i\\
   0.1270 + 0.6797i &  0.9649 + 0.5853i &  0.4218 + 0.9593i &  0.9340 + 0.8143i &  0.7060 + 0.4733i &  0.3171 + 0.7572i &  0.1869 + 0.7792i\\
   0.9134 + 0.6551i &  0.1576 + 0.2238i &  0.9157 + 0.5472i &  0.6787 + 0.2435i &  0.0318 + 0.3517i &  0.9502 + 0.7537i &  0.4898 + 0.9340i\\
   0.6324 + 0.1626i &  0.9706 + 0.7513i &  0.7922 + 0.1386i &  0.7577 + 0.9293i &  0.2769 + 0.8308i &  0.0344 + 0.3804i &  0.4456 + 0.1299i\\
   0.0975 + 0.1190i &  0.9572 + 0.2551i &  0.9595 + 0.1493i &  0.7431 + 0.3500i &  0.0462 + 0.5853i  & 0.4387 + 0.5678i&   0.6463 + 0.5688i\\
   0.2785 + 0.4984i &  0.4854 + 0.5060i &  0.6557 + 0.2575i &  0.3922 + 0.1966i &  0.0971 + 0.5497i &  0.3816 + 0.0759i&   0.7094 + 0.4694i\\
   \end{bmatrix},$$}
{\scriptsize
$$Q=\begin{bmatrix}
-3.9762 &  0.9391 + 0.3070i  & 0.3912 + 0.1498i  & 1.2370 + 0.1579i &  0.3957 + 0.1193i &  0.7923 - 0.1580i  & 1.0187 - 0.8637i\\
   0.9391 - 0.3070i & -3.4741 &  1.5674 - 0.0960i &  0.7959 - 0.8403i&   1.1479 + 0.2799i &  0.5961 + 0.0778i &  0.7059 + 0.1491i\\
   0.3912 - 0.1498i &  1.5674 + 0.0960i & -3.6952 &  1.7877 + 0.1383i &  0.7982 + 0.5997i  & 1.1322 - 0.2553i  & 0.4291 - 0.1804i\\
   1.2370 - 0.1579i &  0.7959 + 0.8403i  & 1.7877 - 0.1383i & -3.9907&   1.5750 - 0.1734i   &1.6866 + 0.1681i   & 1.3819 + 0.4438i\\
   0.3957 - 0.1193i &  1.1479 - 0.2799i   &0.7982 - 0.5997i &  1.5750 + 0.1734i & -3.1372   &1.4904 - 0.1785i &  0.5837 + 0.1395i\\
   0.7923 + 0.1580i &  0.5961 - 0.0778i   &1.1322 + 0.2553i &  1.6866 - 0.1681i  & 1.4904 + 0.1785i  &-2.9003 & 0.2209 - 0.3334i\\
   1.0187 + 0.8637i &  0.7059 - 0.1491i   &0.4291 + 0.1804i &  1.3819 - 0.4438i   &0.5837 - 0.1395i  & 0.2209 + 0.3334i  & -3.5202\\
   \end{bmatrix}.$$}
The eigenvalues of $P(\lambda)$ are $\{-27.0689+7.4062i, 7.8904-0.2781i, -0.6530+4.5349i,-0.7265-4.2218i,
3.1286-0.7004i$, $-2.3956-2.3435i,0.4241+0.9056i,-0.8722-0.4892i,0.3044-0.0681i$,$-0.0311+0.2160i,-0.2133-0.2087i,-0.0396-0.2301i,
-0.0344+0.0094i,0.1266-0.0045i\}$. Note that $|0.4241+0.9056i|=1$. Suppose that we update the eigenvalues $\{(7.8904-0.2781i,0.1266-0.0045i),$ $(3.1268-0.7004i,0.3044-0.0681i), (-27.0689+7.4062i,-0.0344+0.0094i), 0.4241+0.9056i\}$ by $\{\tilde{\lambda}_j\}_{j=1}^7$ which are give by the following two cases, and keep the remaining eigenpairs unchanged.

{\rm (i)} $\tilde{\lambda}_1=-1+2i$ with algebraic multiplicity $3$ and a simple eigenvalue $\tilde{\lambda}_2=-0.6-0.8i$ with $|\tilde{\lambda}_2|=1$.

{\rm (ii)} $\tilde{\lambda}_1=-3-4i$ with algebraic multiplicity $2$ and $\tilde{\lambda}_2=-0.6-0.8i$ with algebraic multiplicity $3$.
\end{example}

By the given information, we can get $Y_1\in\mathbb{C}^{7\times 7}$ and 
\begin{equation*}\Gamma_1=\mbox{diag}\left(\begin{bmatrix}
0 & a_1\\
-\bar{a}_1 & 0\\
\end{bmatrix},\begin{bmatrix}
0 & a_2\\
-\bar{a}_2 & 0\\
\end{bmatrix},\begin{bmatrix}
0 & a_3\\
-\bar{a}_3 & 0\\
\end{bmatrix},-0.1172i\right),\end{equation*}
where $a_1=0.0602+0.3001i$, $a_2=0.1825+0.3725i$, $a_3=-0.1760-0.2313i$. 

For case {\rm (i)}, we have $\tilde{\Lambda}_1=\mbox{diag}\left((-1+2i)I_3+N_3, \frac{1}{-1-2i}I_3+N_3, -0.6-0.8i\right).$ Choose
$\tilde{\Gamma}_1=\mbox{diag}\left(\bigl[\begin{smallmatrix}
0 & S\\
-S^H & 0\\
\end{smallmatrix}\bigr], -i\right)$ randomly with
\begin{equation*}S=\begin{bmatrix}
-12.9944 + 1.4657i &  4.9900 + 0.6688i  & 1.3165 - 1.0627i\\
  24.6515 + 7.6439i&   8.2003 + 2.0782i &  0.0000 - 0.0000i\\
  16.2883 +39.0357i&  -0.0000 - 0.0000i & -0.0000 - 0.0000i\\
\end{bmatrix}.\end{equation*}
Clearly, $||\tilde{\Lambda}_1\tilde{\Gamma}_1\tilde{\Lambda}_1^H-\tilde{\Gamma}_1||_F=8.3622e-16$, i.e., $\tilde{\Gamma}_1\in\mathcal{S}_{(\tilde{\Lambda}_1, H, 1)}$. Taking $\widetilde{\Phi}=I_7$, we can obtain from Algorithm \ref{alg-eep-1} that
{\scriptsize$$\Phi=\begin{bmatrix}
   -0.0133 - 0.0139i&   0.0366 + 0.0148i&   0.0560 - 0.0298i&   0.0000 + 0.0000i&   0.0000 + 0.0000i&   0.0000 + 0.0000i&   0.0000 + 0.0000i\\
   0.0000 + 0.0000i&   0.0000 + 0.0000i&   0.0000 + 0.0000i &  0.0596 - 0.0488i &  0.0031 - 0.0036i & -0.0002 + 0.0006i  & 0.0000 + 0.0000i\\
  -0.1262 - 0.0390i&  -0.1523 - 0.0150i &  0.0464 - 0.0330i &  0.0000 + 0.0000i &  0.0000 + 0.0000i &  0.0000 + 0.0000i   & 0.0000 + 0.0000i\\
   0.0000 + 0.0000i&   0.0000 + 0.0000i  & 0.0000 + 0.0000i &  0.0107 - 0.0067i & -0.1539 + 0.1404i & -0.0018 + 0.0237i &  0.0000 + 0.0000i\\
  -0.1412 + 0.3252i &  0.0632 - 0.2097i   &0.1580 + 0.1818i &  0.0000 + 0.0000i &  0.0000 + 0.0000i &  0.0000 + 0.0000i  & 0.0000 + 0.0000i\\
   0.0000 + 0.0000i &  0.0000 + 0.0000i   &0.0000 + 0.0000i & -0.0032 - 0.0064i &  0.0227 + 0.0492i & -0.4439 - 0.1779i  & 0.0000 + 0.0000i\\
   0.0000 + 0.0000i &  0.0000 + 0.0000i  & 0.0000 + 0.0000i &  0.0000 + 0.0000i  & 0.0000 + 0.0000i &  0.0000 + 0.0000i   & 0.2421 - 0.2421i\\
\end{bmatrix},$$}
which satisfies $||\Phi\tilde{\Gamma}_1\Phi^H-\Gamma_1||_F=4.1309e-15$. Setting $\tilde{Y}_1=Y_1\Phi$, we can obtain from (\ref{gg-10-1}) that  
{\scriptsize
$$\begin{array}{l}
\tilde{A}=\\
\begin{bmatrix}
   3.9731 - 3.3603i & 10.4667 - 5.2824i & -0.7475 +10.0734i&   4.9792 +10.8087i&  -0.9394 + 9.5080i&   6.5600 - 7.0724i & -2.2412 + 3.7615i\\
  -1.0018 + 7.6765i &  1.8450 +11.1626i  &-5.1134 +11.1924i&  -2.8711 +12.5864i&  -7.8130 + 6.4188i&   0.8150 + 8.1950i  & -2.5554 +10.9103i\\
   3.0222 - 5.9529i &  9.5431 - 8.7064i  &-2.7414 + 8.6982i&   4.3144 + 9.4399i&   0.3068 + 8.6394i&   4.0810 -10.4176i  & -5.9583 + 2.7542i\\
  -4.9533 -10.0016i & -6.6949 -15.0639i  &-4.2695 + 1.3770i & -0.7029 - 0.7786i&   2.6778 - 0.2868i&  -9.1765 -13.0755i & -14.0583 - 0.2718i\\
   1.8374 - 3.6541i &  6.4555 - 4.8662i  &-1.4645 + 7.2131i &  3.5681 + 8.3011i&  -0.8681 + 7.2059i&   2.1369 - 6.2912i  & -4.4601 + 3.1713i\\
  -1.0286 - 0.2900i &  1.9066 - 1.2975i  &-0.8447 + 6.0506i  & 1.3509 + 6.1394i&  -1.0129 + 5.1997i&  -0.9899 - 2.1973i  & -2.4116 + 4.2592i\\
  -5.6797 - 1.9316i & -7.8973 - 3.5132i  &-3.2175 - 0.7382i  &-3.6396 - 1.7311i&  -0.1647 - 2.2670i&  -8.5825 - 3.3112i  & -7.5117 - 2.4806i\\
  \end{bmatrix},\\
  \end{array}$$}
{\scriptsize
$$\begin{array}{l}
\tilde{Q}=\\
\begin{bmatrix}
   6.3441 + 0.0000i&   2.7150 - 4.9649i&   2.2839 + 9.5896i&   0.5116 +24.8307i & -1.0046 + 8.1448i&   7.3499 - 1.4858i & -4.9374 + 5.9266i\\
   2.7150 + 4.9649i&   7.6013 + 0.0000i &-12.6865 +15.2991i& -10.7582 +19.6963i &-11.9846 +15.3547i&   7.9935 - 3.1288i  & -6.2772 - 0.5474i\\
   2.2839 - 9.5896i &-12.6865 -15.2991i & -4.5562 - 0.0000i&  -2.9373 +16.2795i &  1.8142 - 2.3419i&  -4.9206 - 7.6461i & -14.8863 + 6.5150i\\
   0.5116 -24.8307i &-10.7582 -19.6963i  &-2.9373 -16.2795i&  -2.4692 - 0.0000i &  1.3169 -12.1059i&  -0.9265 -18.3790i & -10.1310 + 0.8633i\\
  -1.0046 - 8.1448i& -11.9846 -15.3547i   &1.8142 + 2.3419i&   1.3169 +12.1059i & -0.2775 - 0.0000i&  -6.8489 - 8.3369i & -10.6444 + 7.2163i\\
   7.3499 + 1.4858i&   7.9935 + 3.1288i  &-4.9206 + 7.6461i&  -0.9265 +18.3790i & -6.8489 + 8.3369i&   6.5086 + 0.0000i  & -3.7580 + 2.9696i\\
  -4.9374 - 5.9266i & -6.2772 + 0.5474i &-14.8863 - 6.5150i& -10.1310 - 0.8633i &-10.6444 - 7.2163i&  -3.7580 - 2.9696i & -18.2549 - 0.0000i\\
  \end{bmatrix}\\
  \end{array}$$}
 which satisfy $||\tilde{Q}^H-\tilde{Q}||_F=4.5306e-10$ and
\begin{equation*}||\tilde{A}\tilde{Y}_1\tilde{\Lambda}_1^2+\tilde{Q}\tilde{Y}_1\tilde{\Lambda}_1+\tilde{A}^H\tilde{Y}_1||_F=1.7716e-10,\ \ ||\tilde{A}Y_2{\Lambda}_2^2+\tilde{Q}{Y}_2{\Lambda}_2+\tilde{A}^H{Y}_2||_F= 3.1808e-10.\end{equation*}

For case {\rm (ii)}, we have $\tilde{\Lambda}_1=\mbox{diag}\left((-3-4i)I_2+N_2, \frac{1}{-3+4i}I_2+N_2, (-0.6-0.8i)I_3+N_3\right).$ Randomly choose
$\tilde{\Gamma}_1=\mbox{diag}\left(\bigl[\begin{smallmatrix}
0 & S_1\\
-S_1^H & 0\\
\end{smallmatrix}\bigr], S_2\right)$, 
where
\begin{equation*}S_1=\begin{bmatrix}
-0.0046 - 0.0052i & -0.0108 + 0.0385i\\
  -0.9991 + 0.0094i & -0.0000 + 0.0000i\\
\end{bmatrix},\end{equation*}\begin{equation*}S_2=\begin{bmatrix}
 0.0000 + 0.5264i & -0.1634 + 0.1170i &  0.1162 - 0.0339i\\
   0.1634 + 0.1170i &  0.0000 - 0.1210i&  -0.0000 - 0.0000i\\
  -0.1162 - 0.0339i &  0.0000 - 0.0000i &  0.0000 - 0.0000i\\
  \end{bmatrix},\end{equation*}
  which satisfy $||\tilde{\Lambda}_1\tilde{\Gamma}_1\tilde{\Lambda}_1^H-\tilde{\Gamma}_1||_F=1.5017e-15$. Choosing $\widetilde{\Phi}=I_7$, we can obtain from Algorithm \ref{alg-eep-1} that 
  {\scriptsize
  $$\Phi=\begin{bmatrix}
   -0.0038 + 0.0006i&  -0.4248 - 0.3548i&   0.0000 + 0.0000i &  0.0000 + 0.0000i &  0.0000 + 0.0000i &  0.0000 + 0.0000i &  0.0000 + 0.0000i\\
   0.0000 + 0.0000i&   0.0000 + 0.0000i  & 0.4281 - 0.3508i  &-0.0001 - 0.0000i  & 0.0000 + 0.0000i  & 0.0000 + 0.0000i&   0.0000 + 0.0000i\\
   3.0916 - 0.9061i&  -0.0189 - 0.0121i   &0.0000 + 0.0000i  & 0.0000 + 0.0000i  & 0.0000 + 0.0000i  & 0.0000 + 0.0000i &  0.0000 + 0.0000i\\
   0.0000 + 0.0000i&   0.0000 + 0.0000i   &0.0008 - 0.0005i  & 2.9084 + 1.3859i  & 0.0000 + 0.0000i  & 0.0000 + 0.0000i  & 0.0000 + 0.0000i\\
   0.0000 + 0.0000i&   0.0000 + 0.0000i   &0.0000 + 0.0000i  & 0.0000 + 0.0000i  &-0.6215 - 0.5720i  & 0.1047 - 0.8498i   & 1.2544 - 2.5307i\\
   0.0000 + 0.0000i&   0.0000 + 0.0000i   &0.0000 + 0.0000i  & 0.0000 + 0.0000i  & 0.0144 + 0.0809i  & 0.4191 - 0.4295i &  2.6966 - 1.3365i\\
   0.0000 + 0.0000i&   0.0000 + 0.0000i   &0.0000 + 0.0000i &  0.0000 + 0.0000i  &-0.1138 - 0.2075i  &-0.3223 + 0.6710i  & 0.1095 - 0.1095i\\ 
  \end{bmatrix},$$}
   with $||\Phi\tilde{\Gamma}_1\Phi^H-\Gamma_1||_F=4.0322e-15$. Then,  we can obtain from (\ref{gg-10-1}) that 
 {\scriptsize$$\tilde{A}=\begin{bmatrix}
  -3.5026 + 3.6133i&  -3.7588 + 2.7466i&  -0.1126 + 0.3782i&  -2.7692 - 2.5176i & -3.2182 + 1.7159i&  -0.7378 + 1.4255i & -7.5387 - 1.7600i\\
  -0.8489 + 2.3269i&  -1.5242 + 0.7422i &  0.7242 - 0.1038i&  -0.5180 - 2.5260i & -0.2803 + 1.9143i&  -0.7154 - 0.0585i & -4.8161 - 0.3175i\\
  -1.2449 + 3.1440i&  -0.4500 + 1.6868i &  1.7114 - 0.0020i&   0.1440 - 2.5983i & -0.3231 + 2.4856i&   0.6219 + 0.6281i & -4.5622 - 1.0591i\\
   4.4939 + 2.2527i&   4.0193 + 0.3024i &  5.6040 - 0.7309i&   3.4126 - 2.2799i &  3.3946 + 3.3362i&   4.6725 + 0.1580i &  0.6929 + 1.2370i\\
  -1.2953 + 3.3783i&  -1.0905 + 3.1088i &  0.9869 + 0.0522i&  -1.0982 - 1.2876i & -1.8119 + 3.0295i&  -0.2524 + 1.3682i & -5.8284 + 0.3224i\\
  -0.9100 + 2.5178i&  -0.6073 + 1.2106i &  1.5588 - 0.4372i&  -0.7112 - 1.5569i & -0.0439 + 2.5169i&  -0.5693 + 0.1953i & -3.6159 - 0.1417i\\
   1.2678 + 5.8652i&   1.5612 + 4.9450i &  2.8567 + 3.1627i&   0.9525 + 2.1122i &  0.2416 + 5.5797i&   0.8277 + 3.1132i & -3.5745 + 3.8444i\\
 \end{bmatrix},$$}
 {\scriptsize$$\tilde{Q}=\begin{bmatrix}
 -13.5111 + 0.0000i&  -5.3083 - 1.4341i&  -3.0726 - 1.9760i &  1.2516 - 7.2103i & -6.2517 - 1.8733i&  -2.9823 - 2.6150i & -6.2992 -11.0602i\\
  -5.3083 + 1.4341i&  -7.3211 + 0.0000i&   1.2178 - 0.3156i &  3.1580 - 6.9580i & -3.4480 + 0.4578i&  -1.5032 - 0.3485i & -5.9977 - 8.1041i\\
  -3.0726 + 1.9760i&   1.2178 + 0.3156i&  -0.8639 - 0.0000i &  8.2049 - 4.4488i & -1.0194 + 1.9211i&   2.1085 - 0.3021i & -2.0236 - 7.1315i\\
   1.2516 + 7.2103i&   3.1580 + 6.9580i&   8.2049 + 4.4488i &  5.5315 + 0.0000i &  1.5643 + 5.9013i&   4.8997 + 4.6812i & -0.2164 - 2.4030i\\
  -6.2517 + 1.8733i&  -3.4480 - 0.4578i&  -1.0194 - 1.9211i &  1.5643 - 5.9013i & -7.5692 + 0.0000i&  -0.9484 - 1.4467i & -6.5565 - 6.8059i\\
  -2.9823 + 2.6150i&  -1.5032 + 0.3485i&   2.1085 + 0.3021i &  4.8997 - 4.6812i & -0.9484 + 1.4467i&  -3.9143 - 0.0000i & -4.3014 - 5.8846i\\
  -6.2992 +11.0602i&  -5.9977 + 8.1041i&  -2.0236 + 7.1315i & -0.2164 + 2.4030i & -6.5565 + 6.8059i&  -4.3014 + 5.8846i& -13.1183 + 0.0000i\\ 
  \end{bmatrix},$$}
  with $||\tilde{Q}-\tilde{Q}^H||_F=1.5169e-10$, which satisfy
 \begin{equation*}||\tilde{A}\tilde{Y}_1\tilde{\Lambda}_1^2+\tilde{Q}\tilde{Y}_1\tilde{\Lambda}_1+\tilde{A}^H\tilde{Y}_1||_F=4.0668e-10,\ \ ||\tilde{A}Y_2{\Lambda}_2^2+\tilde{Q}{Y}_2{\Lambda}_2+\tilde{A}^H{Y}_2||_F=1.2003e-10.\end{equation*}

We can see from Example \ref{exa-eep3}  that all the prescribed eigenvalues can be reproduced accurately  by the updated system, and the remaining eigenpairs are kept unchanged, which shows that the Algorithm \ref{alg-eep-1} is effective for the multiple eigenvalues. Moreover, for the $(H,\epsilon)$-palindromic system, the pair of eigenvalues $(\lambda, \frac{1}{\bar{\lambda}})$ with $|\lambda|\neq 1$ can be replaced by two model one eigenvalues. 
 
\section{Conclusions}
In this paper, we have derived the spectral decomposition of the $(\star, \epsilon)$-structured palindromic quadratic matrix polynomial in the unified form. With a standard pair of $P(\lambda)$, the structures of the parameter matrix $\Gamma$ are provided in the Theorem \ref{thm-gamma} and Theorem \ref{thm-gamma-2}. When all eigenvalues of $P(\lambda)$ are semi-simple, $\Gamma$ has some special simpler forms which are given by the Theorem \ref{thm-s-1} and Theorem \ref{thm-s-2}. Based on the spectral decomposition,  analytical solutions of the no spill-over MUP of the $(\star, \epsilon)$-palindromic system are characterized. Numerical experiments demonstrate that our proposed algorithms are effective for both simple and multiple eigenvalues.




\section*{Declaration of  competing interest}
The authors declare that they have no known competing financial interests or personal relationships that could have appeared to influence the work reported in this paper.


 \end{document}